\documentclass[a4paper,12pt]{amsart}
\usepackage{amsmath,amssymb,amsfonts,amsthm}
\usepackage[abbrev,nobysame]{amsrefs}
\usepackage{mathtools,cases}
\usepackage{mathrsfs}
\usepackage{hyperref}
\usepackage[]{todonotes}
\usepackage[all]{xy}
\numberwithin{equation}{section}
\usepackage{smartdiagram}

\hypersetup{
 hidelinks,
 bookmarksopen=true
 }

\numberwithin{equation}{section}

\newcommand{\p}{\partial}
\newcommand{\vphi}{\varphi}
\newcommand{\om}{\omega}
\newcommand{\tri}{\triangle}
\newcommand{\eps}{\epsilon}
\newcommand{\thmref}[1]{Theorem~\ref{#1}}

\newcommand{\lemref}[1]{Lemma~\ref{#1}}
\newcommand{\corref}[1]{Corollary~\ref{#1}}

\newcommand{\Om}{\Omega}

\def\p{\partial}
\def\b{\beta}

\DeclareMathOperator{\tr}{Tr}

\newtheorem{theorem}{Theorem}[section]
\newtheorem{thm}[theorem]{Theorem}
\newtheorem{rem}[theorem]{Remark}
\newtheorem{defn}[theorem]{Definition}
\newtheorem{lemma}[theorem]{Lemma}
\newtheorem{lem}[theorem]{Lemma}

\newtheorem{proposition}[theorem]{Proposition}
\newtheorem{prop}[theorem]{Proposition}
\newtheorem{conj}[theorem]{Conjecture}

\newtheorem{cor}[theorem]{Corollary}
\newtheorem{ques}[theorem]{Question}

\newtheorem{problem}[theorem]{Problem}

\def\tr{\mathop{\rm tr}\nolimits}

\let\ul=\underline

\let\vphi=\varphi

\title[On uniform log $K$-stability for cscK cone metrics]{On uniform log $K$-stability for constant scalar curvature K\"ahler cone metrics}

\author{Takahiro Aoi}
  \address{Osaka Prefectural Abuno High School, Osaka 569-1141, JAPAN}
  \email[Takahiro Aoi]{takahiro.aoi.math@gmail.com}

\author{Yoshinori Hashimoto}
  \address{Tokyo Institute of Technology, Tokyo 152-8551, JAPAN}
  \email[Yoshinori Hashimoto]{hashimoto@math.titech.ac.jp}
  
\author{Kai Zheng}
  \address{Tongji University, Shanghai 200092, P.R. CHINA}
  \email[Kai Zheng]{KZheng@tongji.edu.cn}

\date{\today}

\begin{document}
\maketitle

\begin{abstract} We prove that the existence of constant scalar curvature K\"ahler metrics with cone singularities along a divisor implies log $K$-polystability and $G$-uniform log $K$-stability, where $G$ is the automorphism group which preserves the divisor. We also show that a constant scalar curvature K\"ahler cone metric along an ample divisor of sufficiently large degree always exists. We further show several properties of the path of constant scalar curvature K\"ahler cone metrics and discuss uniform log $K$-stability of normal varieties.
\end{abstract}

\tableofcontents

\section{Introduction}

In K\"ahler geometry, the existence of canonical metrics is conjectured to be equivalent to algebro-geometric stability. Geodesic stability is used by Chen--Cheng \cite{Chen Cheng 2} in the recent proof of Donaldson's geodesic stability conjecture for the existence of constant scalar curvature K\"ahler (cscK) metric. In the logarithmic (log) setting, the uniqueness of cscK cone (cscKc for short) metric was proven in \cite{MR4020314}. Later, in \cite{arXiv:1803.09506}, several existence results for cscK cone metrics were proven, extending existence results in \cite{Chen Cheng 2}. In particular, Theorem 1.8 in \cite{arXiv:1803.09506} shows that the existence of cscK cone metric is equivalent to log geodesic stability, see Section \ref{Log geodesic stability} for the definition.

In this article, we further study the log Yau--Tian--Donaldson (YTD) conjecture for cscK cone metrics in terms of various stability notions, to extend most of the stability results to the log setting as continuation of \cite{MR4020314,arXiv:1803.09506}. 

$K$-stability was defined by Tian \cite{MR894378} and generalised by Donaldson in \cite{MR1988506}.  Log $K$-stability was introduced by Donaldson in \cite{MR2975584}. The log $K$-stability has a natural connection to the minimal model program in birational geometry, which has been studied deeply in \cite{MR3403730,MR3194814} etc.

Let $X$ be a K\"ahler manifold, and $D$ be a smooth effective divisor on $X$ (we also consider the case $D$ is a simple normal crossing divisor, but for the most of the paper $D$ is smooth). We further set $L$ to be an ample line bundle and denote by $\Om$ the K\"ahler class associated with $L$. The pair of $X$ and $D$, together with the polarisation $L$ will be called a \textbf{polarised pair} and denoted as $((X,L);D)$.
\begin{conj}[Log YTD conjecture]\label{log YTD necessary}
The polarised pair $((X,L);D)$ admits a cscK cone metric if and only if it is log $K$-polystable.
\end{conj}

The precise definition of log $K$-polystability is given in Section \ref{Test configurations and log K stability} and the definition of cscK cone metrics was given in \cite[Definition 3.1]{MR4020314}. We let $\beta$ be a number in $(0,1]$. A \emph{K\"ahler cone metric} $\om$ of cone angle $2\pi\b$ along the divisor $D$,
is a smooth K\"ahler metric on the regular part $M:=X\setminus D$, and quasi-isometric to the cone flat metric
in the coordinate chart 
centred at a point on $D$.
Roughly speaking, a \emph{cscK cone metric} of cone angle $2\pi\b$ along the divisor $D$ is defined to be a K\"ahler cone metric, which has constant scalar curvature outside $D$.


Let $G$ be the identity component of the group of holomorphic automorphisms of $X$ which fix the divisor $D$. 
The uniqueness of cscK cone metrics is proven in \cite{MR4020314}, that is
	the cscK cone metric is unique up to automorphisms in $G$. Combining the uniqueness result with an asymptotic formula of the log $K$-energy \cite{MR3985614}, we will show in (\ref{log K semistability}) of \thmref{log K} that

\begin{thm}[\thmref{log K}]\label{introduction log K semistability}
	Suppose that $((X,L);D)$ admits a cscK cone metric of angle $2 \pi \beta$. Then $((X,L);D)$ is log $K$-semistable with angle $2 \pi \beta$.
\end{thm}
 \begin{rem}
 When $\beta=1$, \thmref{introduction log K semistability} was proven in Donaldson \cite{MR2192937} by making use of Kodaira embedding and asymptotic expansion of Bergman kernel.  Donaldson's method is extended to twisted cscK metric by Dervan \cite{MR3564626}. 
 \end{rem}
 \begin{rem}Donaldson's semistability result \cite{MR2192937} is extended to orbifolds by Ross--Thomas in \cite{MR2819757}, where they also ask whether this result could be further extended to cscK metrics with cone singularities. \thmref{introduction log K semistability} provides an positive answer to their question, in the case when the cone singularities are formed along a divisor, by using a completely different approach. 
 \end{rem}

 \begin{rem}\label{Log Kodaira embedding}
The proof of \cite{MR2819757} embeds polarised orbifolds into weighted projective space via a weighted version of Kodaira embedding. Similar idea was expected for general cone metrics. However, it is only successful in very limited case, e.g. on a Riemann surface \cite{MR4267616,MR4238257} and the projective completion of an ample line bundle over a cscK base \cite{MR4023378}. \thmref{introduction log K semistability} is also related to \cite[Conjecture 1.1]{MR4267616}.
 \end{rem}

 \begin{rem}
 We believe that \thmref{introduction log K semistability} also holds for a general compact K\"ahler manifold, with the K\"ahler class that is not associated to an ample line bundle, by following the argument \cite{MR3880285,MR3696600,MR4095424,MR3881961}, but in this paper we decide not to discuss the details which could be technical.
 \end{rem}

It is shown in \cite{arXiv:1803.09506} that the existence of a cscK cone metric on $(X,L)$ is equivalent to the $d_{1,G}$-coercivity of the log $K$-energy. With the help of this existence result, \thmref{introduction log K semistability} is strengthened to log $K$-polystability. As a result, we prove one direction of Conjecture \ref{log YTD necessary}.
\begin{thm}[\thmref{log K}]\label{introduction log K stability}
	Suppose that $((X,L);D)$ admits a cscK cone metric of angle $2 \pi \beta$. Then $((X,L);D)$ is log $K$-polystable with angle $2 \pi \beta$.
\end{thm}

The $K$-stability condition is also expected to be replaced by uniform stability.

\begin{conj}[Log YTD conjecture for uniform stability]\label{Uniform log YTD necessary}
The polarised pair $((X,L);D)$ admits a cscK cone metric if and only if it is uniformly log $K$-stable.
\end{conj}
This existence result, together with \thmref{cone properness theorem}, implies uniform log $K$-stability, see Section \ref{Uniform log K stability}.
\begin{thm}[\thmref{log K}]\label{introduction uniform log K stability}
	Suppose that $((X,L);D)$ admits a cscK cone metric of angle $2 \pi \beta$. Then $((X,L);D)$ is $G$-uniformly log $K$-stable with angle $2 \pi \beta$, with $G = \mathrm{Aut}_0 ((X,L);D)$.
\end{thm}

While the above results hold for a general polarised pair $((X,L);D)$, we get more precise results when the cohomology class of $D$ is a multiple of $C_1(L)$. One result in this direction is the following, which ensures that the cscK cone metrics always \textit{exist} when the divisor $D$ is ample and of sufficiently large degree.

\begin{thm}[\thmref{thmalphainv}] \label{thmalphainvint}
	For any $0 < \beta < 1$, there exists $m_0 \in \mathbb{N}$ which depends only on $X$, $L$, and $\beta$, such that $(X,L)$ admits a constant scalar curvature K\"ahler metric with cone singularities of cone angle $2 \pi \beta$ along a generic member of the linear system $|mL|$ if $m \ge m_0 $.
\end{thm}
Further precise quantitative relation between the multiplicity $m$ and the cone angle $\beta$ is given in \thmref{small angle existence}.

When we have $D \in |mL|$ and $m$ is large enough, as above, it turns out that the average value $\underline{S}^D$ of the scalar curvature of K\"ahler metrics in $C_1 (L |_D )$ over $D$ satisfies $\underline{S}^D \ll 0 < \beta$. We have a result on the log $K$-\textit{instability} in the situation that is somewhat complementary to the above.

\begin{thm}[\thmref{destabilizing test configuration}]
Suppose that $D \in |L|$ is smooth. Then $((X,L);D)$ is log $K$-unstable with angle $2 \pi \beta$ if $\beta$ satisfies 
\begin{equation*}
\beta < \frac{\underline{S}^D}{n(n-1)}.
\end{equation*}
\end{thm}

We also prove various sufficient conditions for the uniform log $K$-stability. We believe that comparing the theorem below with \thmref{thmalphainvint} (or more precisely \eqref{cohomology conditions Kahler} that is used in its proof) may give an interesting observation for the log YTD conjecture.

\begin{thm}[\thmref{topological conditions to beta conditions}]
Suppose that $((X,L);D)$ is a smooth polarised pair such that $D \in |mL|$ for some $m \in \mathbb{N}$, such that it satisfies
\begin{align*}
\ul S_1\leq mn\text{ and }(n+1)\lambda\leq \ul S_1+m,
\end{align*}
where $\ul S_1$ is the average scalar curvature of $(X,L)$ and $\lambda \in \mathbb{R}$ is the nef threshold of $L$ (see Definition \ref{Lambda}).

Then $((X,L);D)$ is uniformly log $K$-stable with cone angle $2\pi\beta$ satisfying the constraint 
\begin{align*}
1-\frac{(n+1)\lambda-\ul S_1}{m}\leq \beta<\beta_u,
\end{align*}
where $\beta_u$ is a constant defined in terms of the log alpha invariant and $m$ as in Definition \ref{beta u}.
\end{thm}

We also present some results for the uniform log $K$-stability for singular varieties. Suppose that $X$ is a $\mathbb{Q}$-Gorenstein normal projective variety, $L$ is an ample Cartier divisor on $X$, and $\tri$ is an effective integral reduced Cartier divisor on $X$. We also write $\ul S_{\beta}$ for the constant defined by \eqref{intersection averaged scalar}, which is the average log scalar curvature when $X$ and $\tri$ are smooth. We have the following results.

\begin{thm}[\thmref{log canonical}]
If $(X , (1 - \beta ) \tri )$ is log canonical, we have the following.
\begin{itemize}
\item Suppose $\ul S_\beta< 0$ and that there exists $\eta \ge 0$ such that
\begin{align*}
\left\{
\begin{array}{lcl}
	&(i)&0\leq \eta<\frac{n+1}{n}\alpha_\beta,\\
	&(ii)& \eta L + K_X + (1 - \beta ) \tri \text{ is ample},\\
	&(iii)& -(n-1) (K_X + (1 - \beta ) \tri) - (\ul S_\beta-\eta) L \text{ is ample},
\end{array}
\right.
\end{align*} 
where $\alpha_{\beta}$ is the log alpha invariant in Definition \ref{aglog alpha invariant}. Then $((X,L);\tri)$ is uniformly log $K$-stable with angle $2 \pi \beta$.
\item Suppose that $\ul S_\beta<(n+1)\alpha_\beta$, and $-\ul S_\beta L- (n+1)(K_X + (1 - \beta ) \tri)$ is nef. Then $((X,L);\tri)$ is uniformly log $K$-stable with angle $2 \pi \beta$.
\end{itemize}
\end{thm}

The second item above can be seen as a variant of Dervan's result \cite{MR3428958,MR3564626}; see Remark \ref{remcpdervan} for more details. We also prove the following result, which partially generalises the result of Odaka--Sun \cite[Theorem 6.1]{MR3403730} by relaxing the hypothesis on the polarisation $L$.

\begin{thm}[\thmref{Kawamata log terminal inverse}]
Suppose that
\begin{itemize}
	\item $(X, (1 - \beta) \tri)$ is log $\mathbb{Q}$-Fano, i.e.~$-K_X - (1 - \beta ) \tri$ is ample, and
	\item $\ul S_\beta L - n (-K_X + (1 - \beta ) \tri)$ is nef.
\end{itemize}
Then $(X, (1 - \beta) \tri)$ is Kawamata log terminal if $((X,L);\tri)$ is log $K$-semistable with angle $2 \pi \beta$.
\end{thm}

\bigskip

\noindent \textbf{Organisation of the paper.} We review the basic materials on the cscK cone metrics, its variational characterisation, and the automorphism group and the Futaki invariant in Section \ref{Log K stability}. Section \ref{scrvlst} is devoted to the review of various stability notions that are important in this paper. In particular, we recall various log $K$-stabilities in algebraic geometry, as well as the log properness result on the analytic side which was proved in \cite{arXiv:1803.09506} and plays a crucially important role in this paper. In Section \ref{scrsostab} we prove one of our main results \thmref{log K}, by applying the important results of Boucksom--Hisamoto--Jonsson \cite{MR3669511,MR3985614} and the log properness result \cite{arXiv:1803.09506}. We review the alpha invariant in Section \ref{scrsoexccm}, and prove \thmref{thmalphainv} as an application of it, as well as other existence results. We finally discuss the uniform log $K$-stability for a possibly singular normal variety in Section \ref{scnormvar}.

\bigskip

\noindent \textbf{Notation.} We comment on the notational conventions that we use in this paper. 

In this paper we shall mainly consider a smooth polarised pair, denoted by $((X,L);D)$, but many topics extend to the case of a more general compact K\"ahler manifold $X$ with a smooth effective divisor $D$, with the fixed K\"ahler class $\Omega$; the extension is sometimes straightforward, but there are nontrivial cases in which the extension to general K\"ahler manifolds is only conjectural (at least to the best of the authors' knowledge). The convention that we use is as follows: we write $\Omega$ for the K\"ahler class, which is equal to $C_1(L)$ when the manifold $X$ under consideration is a polarised smooth projective variety, but $\Omega$ stands for a general K\"ahler class when it is explicitly stated that $X$ may not be projective.

We shall use the additive notation for the tensor product of line bundles, unless explicitly stated otherwise. We consistently write $L_D$ for the line bundle $\mathcal{O}_X (D)$ on $X$ defined by the divisor $D$.

While $X$ is supposed to be smooth in much of this paper, we also consider a pair $(X , \tri)$ where $X$ and $\tri$ may also be singular. See Section \ref{scnormvar} for more details on the singularities of $(X , \tri)$.

\medskip

\noindent \textbf{Acknowledgements.} The first author would like to thank Professor Ryoichi Kobayashi for many helpful comments. The second author is partially supported by JSPS KAKENHI (Grant-in-Aid for Early-Career Scientists), Grant Number JP19K14524, and thanks Ruadha\'i Dervan and Julien Keller for helpful comments. The third author is partially supported by the Fundamental Research Funds for the Central Universities (Grant Numbers: 22120200041, 22120210282). The third author would like to thank Gao Chen, Jingzhou Sun, Song Sun for helpful comments.

\section{CscK cone metrics}\label{Log K stability}

Let $X$ be a (possibly non-projective) compact K\"ahler manifold and $\Omega$ be a K\"ahler class, and $D$ be a smooth effective divisor in $X$. In this section, we will review recent progress on cscK cone metrics \cite{MR4020314,arXiv:1803.09506}. We also obtain many properties of corresponding energy functionals and the cscK cone path, which extend paralleling results on K\"ahler--Einstein cone metrics.

\subsection{Definition of cscK cone metrics}
We recall the definition of cscK cone metrics. We let 
\begin{align*}
C_1(X,D):=C_1(X)-(1-\b)C_1(L_D).
\end{align*}
and choose a smooth form $\theta\in C_1(X,D)$. We let $s$ be the defining section of $D$ and $h$ is a Hermitian metric on the associated line bundle $L_D$. We denote  by $\Theta_D$ the curvature form
\begin{align*}
\Theta_D=-i\p\bar\p\log h.
\end{align*}

Since $\theta\in C_1(X,D)$, by cohomology condition, there exists a smooth function $h_0$ such that 
\begin{align}\label{h0}
Ric(\om_0)=\theta+(1-\beta)\Theta_D+i\p\bar\p h_0.
\end{align}
The reference K\"ahler cone metric $\om_\theta\in \Omega$ is obtained by solving the following equation
\begin{align}\label{Rictheta}
Ric(\om_\theta)=\theta+2\pi (1-\b)[D],
\end{align}
equivalently,
\begin{align}\label{Rictheta pde}
\om_\theta^n=e^{h_0}|s|_h^{2\beta-2}\om_0^n ,
\end{align}
and a normalisation condition $\int_X \om^n_\theta=V$.



\begin{defn}\emph{(\cite[Definition 3.1]{MR4020314})}\label{csckconemetricdefn} A \textbf{cscK cone metric} in $\Omega$
	is defined to be a solution to the coupled system
\begin{equation}
\frac{\om_{cscK}^n}{\om_\theta^n}=e^F,\quad
\tri_{\om_{cscK}} F=\tr_{\om_{cscK}}\theta-\underline S_\b.
\end{equation} 
Here, the constant $\underline S_\b$ is a topological constant, equal to
\begin{equation} \label{eqavlgsc}
\underline S_\b= n \frac{\int_X C_1(X,D)\Om^{n-1}}{\int_X \Om^n}=\frac{n\int_X\theta\wedge \om_0^{n-1}}{V},
\end{equation}
with the volume $V=\int_X\om_0^{n}$.
\end{defn}
Putting the formula of $F$ from the first equation to the second equation, we could see that the scalar curvature of $\om_{cscK}$ equals to the constant $\underline S_\b$ on the regular part $X\setminus D$.

\begin{rem}
	The cone metric and the cscK cone metric can be described more explicitly in terms of the moment polytope when $X$ is toric; see e.g.~\cite[Section 2]{arXiv:2005.03502} for more details.
\end{rem}

\subsection{CscK cone path}
\begin{defn}\emph{(\cite[Definition 8.1]{arXiv:1803.09506})}
The \textbf{cscK cone path} is defined to be the one-parameter $\beta$ family of cscK cone metrics with cone angle $2\pi\beta$.
\end{defn}
Many constructions of cscK cone metrics with small positive cone angle were given in \cite{MR3857693}.

The existence result (see \thmref{cone properness theorem}) leads to the following openness of the cscK cone path, when there are no nontrivial holomorphic vector fields which preserve $D$ (i.e.~$\mathfrak{aut} ((X,L);D) = 0$ in the notation of Definition \ref{dfautxld}).

\begin{thm}\label{openess}\emph{(\cite[Theorem 8.2]{arXiv:1803.09506})}
Suppose that there are no nontrivial holomorphic vector fields which preserve $D$. The cscK cone path is open when $0<\beta\leq1$. Precisely, if there exists a cscK cone metric with cone angle $2\pi\beta_0\in (0,2\pi)$, then there is a constant $\delta>0$ such that for all $\beta\in (\beta_0-\delta,\beta_0+\delta)$, there exists a cscK cone metric with cone angle $2\pi\beta$. Especially, when $\beta=1$, then interval becomes $(1-\delta,1]$.
\end{thm}

The hypothesis on the automorphism group is indeed necessary; see \cite[Theorem 1.7 and Remark 1.8]{MR3978320} and also Proposition \ref{ppcnaglfut}.

As a result, we recall the definition of the maximal cone angle of the cscK cone metric.
\begin{defn}\label{maximal existence cone}
\begin{align*}
\mathfrak \beta_{cscKc}((X,\Omega);D):=\sup_{0<\beta\leq 1}\{\exists \text{ a cscK cone metric with cone angle }2\pi\beta\}.
\end{align*}
\end{defn}

\subsection{Energy functionals in the log setting}\label{Log energy functionals}
\subsubsection{Space of K\"ahler cone potentials}


%

\begin{defn}
We set $\mathcal H_\beta(\om_0)$ to be the \textbf{space of all K\"ahler cone potentials} $\vphi$ such that $\om_\vphi:=\om_0+i\p\bar\p\vphi$ is a K\"ahler cone metric of angle $2 \pi \beta$. 
\end{defn}

Similarly, we write $\mathcal{H} (\omega_0)$ for the space of K\"ahler potentials that are smooth globally over $X$.
Let $V=\int_X\om_0^n$.
The $d_1$-distance on $\mathcal{H} (\omega_0)$ is defined to be the norm $V^{-1}\int_X|\psi|\om_\vphi^n$ over the tangent space $T\mathcal H$ of $\mathcal H$. The definition is the same for $\mathcal H_\beta$.
It is important to note that there exists a completion of the above space $\mathcal{H}$, which is denoted by $\mathcal{E}^1$. 
\begin{defn}\label{E 1}
A $\om_0$-psh function $\varphi\in PSH(X,\omega_0)$ is an element in the space $\mathcal E^1(X,\om_0)$, if
\begin{itemize}
\item $\varphi$ has \textit{full Monge--Amp\`ere mass}, i.e.
\begin{align*}
\lim_{j\rightarrow\infty}\int_{\{\varphi>-j\}}(\omega_0+i\p\bar\p\max\{\varphi,-j\})^n=V.
\end{align*}
Then the Monge--Amp\`ere operator is well-defined for such $\vphi$.
\item
$
 \vphi\in L^1(\om_0+i\p\bar\p\vphi).
$
\end{itemize}
\end{defn}
The reader is referred to the lecture notes \cite{MR3996485} by Darvas or the monograph \cite{MR3617346} by Guedj--Zeriahi for more details and many important properties of $\mathcal{E}^1$. A particularly important observation for our paper is that we have $\mathcal H_\beta(\om_0) \subset \mathcal{E}^1$ for all $0 < \beta \leq 1$, according to the Definition \ref{E 1}.

Actually, $\mathcal{E}^1$ is also the completion of $\mathcal H_\beta(\om_0)$ under the $d_1$-distance.
This could be seen by using Donaldson's model cone metric $\om_0+\delta i\p\bar\p |s|^{2\beta}_h$, here $\delta$ is a small positive constant, $s$ is a defining section of $D$ and $h$ is an Hermitian metric on $L_D$. Then we see that for any given K\"ahler potential $\vphi$, $\vphi_j=\vphi+ j^{-1}|s|^{2\beta}_h$ is a K\"ahler cone metric for sufficiently large $j$. The sequence $\vphi_j$ decreases to $\vphi$ and converges in $d_1$-distance, which means any smooth K\"ahler potential is approximated by a sequence of K\"ahler cone potential and the $d_1$-completion of $\mathcal H_\beta(\om_0)$ is also $\mathcal E^1$. 
\subsubsection{Log entropy}
\begin{defn}[Log entropy]
The \textbf{log entropy} on $\mathcal H_\b$ is defined to be
\begin{align}\label{entropy}
E_\beta(\vphi):=\frac{1}{V}\int_M\log\frac{\om^n_\vphi}{\om_0^n|s|_h^{2\b-2}e^{h_0}}\om_\vphi^{n},
\end{align}
\end{defn}
When $\beta=1$, it coincides with the classical entropy functional of K\"ahler metric $\om_\vphi$.
\subsubsection{$D$-functional}
The Euler--Lagrange functional for the Monge--Amper\`e operator is 
\begin{align}
D_{\om_0}(\vphi)
&:=\frac{1}{V}\frac{1}{n+1}\sum_{j=0}^{n}\int_{M}\vphi\om_0^{j}\wedge\om_{\varphi}^{n-j}, \label{defmaeng}
\end{align}
since direct computation shows that its first variation is
\begin{align*}
\p_t D_{\om_0}(\vphi)=\frac{1}{V}\int_M \p_t\vphi \om_\vphi^n.
\end{align*}
As a result, we also obtain another expression of $D_{\omega_0}$,
\begin{align*}
D_{\om_0}(\vphi)
&:=\frac{1}{V}\int_0^1\int_M \p_t\vphi \om_\vphi^ndt.
\end{align*}
\subsubsection{$j$-functional}
Let $\chi$ to be a closed $(1,1)$-form. 
\begin{defn}[$J_{\chi}$-functional]\label{J functional}
The \textbf{log $J_\chi$-functional} is defined to be
\begin{align*}
J_{\chi}(\vphi)&:=j_{\chi}(\vphi)-\ul{\chi}\cdot D_{\om_0}(\vphi).
\end{align*}
The \textbf{log $j_\chi$-functional} is defined to be
\begin{align*}
j_{\chi}(\vphi)&:=\frac{1}{V}\int_{M}\vphi
\sum_{j=0}^{n-1}\om_0^{j}\wedge
\om_\vphi^{n-1-j}\wedge \chi.
\end{align*}
Here $\ul{\chi}$ is the average of $\chi$ as follows
\begin{align*}
\ul{\chi}
:=\frac{\int_X [\chi]\Om^{n-1}}{\int_X \Om^n}=\frac{n\int_X\chi\wedge \om_0^{n-1}}{V}.
\end{align*}

\end{defn}
It is direct to see that its first variation is
\begin{align}\label{derivatives of J chi}
\p_t J_{\chi}(\vphi)=\frac{1}{V}\int_{M}\p_t\vphi
(n\chi\wedge \om_\vphi^{n-1}-\underline\chi\om^n_\vphi).
\end{align}

We further define \textbf{Aubin's $J$-functional} as
\begin{equation} \label{defaubjfc}
J^A_{\om_0}(\vphi):=\frac{1}{V}\int_M\vphi\om_0^n-D_{\om_0}(\vphi).
\end{equation}

\subsubsection{Log $K$-energy}
\begin{defn}\label{logKenergy}\emph{(Log $K$-energy \cite[Equation 3.9]{MR4020314})}
The \textbf{log $K$-energy} is defined on $\mathcal H_\b$ as
\begin{align}\label{log K energy}
\nu_\beta(\vphi)
&:=E_\beta(\vphi)
+J_{-\theta}(\vphi)+\frac{1}{V}\int_M (\mathfrak h+h_0)\om_0^n,
\end{align}
where $
\mathfrak h:=-(1-\b)\log |s|_h^2.
$
\end{defn}

When $\beta=1$, we choose $\theta=Ric(\omega_0)$. Then the log $K$-energy coincides with Mabuchi's $K$-energy, which will be denoted by,
\begin{align}\label{Mabuchienergy}
\nu_1(\vphi)
=E_1(\vphi)+
j_{-Ric(\om_0)}(\vphi)+\ul{S}_1\cdot D_{\om_0}(\vphi),
\end{align}
with the entropy term
\begin{align*}
E_1(\vphi)=\frac{1}{V}\int_M\log\frac{\om^n_\vphi}{\om_0^n}\om_\vphi^{n}.
\end{align*}

We denote $\nu_\beta(\omega_0,\omega_\vphi):=\nu_\beta(\vphi)$.
\begin{prop}
The log $K$-energy satisfies the co-cycle condition
\begin{align*}
\nu_\beta(\omega_1,\omega_3)=\nu_\beta(\omega_1,\omega_2)+\nu_\beta(\omega_2,\omega_3).
\end{align*}
\end{prop}
The critical point of the log $K$-energy is the cscK cone metric (\cite[Lemma 3.5]{MR4020314}).

The following identity between log $K$-energy and $K$-energy was given in \cite{MR4020314}, we record it here.
\begin{prop}The relation between the log $K$-energy and the $K$-energy is given by
\begin{align*}
\nu_\beta(\vphi)&=\nu_1(\vphi)+(1-\beta)\frac{1}{V}\int_M\log|s|^2_h(\omega^n_\vphi-\om^n)\\
&+(1-\b)j_{\Theta_D}(\vphi) 
-(1-\b)\frac{C_1(L_D)\Om^{n-1}}{\Om^n}\cdot D_{\om_0}(\vphi).
\end{align*}
\end{prop}
\begin{proof}
Firstly, the entropy term is simplified to be
\begin{align*}
E_\beta(\vphi)&=\frac{1}{V}\int_M\log\frac{\om^n_\vphi}{\om_0^n|s|_h^{2\b-2}e^{h_0}}\om_\vphi^{n}\\
&=E_1(\vphi)-\frac{1}{V}\int_M\log(|s|_h^{2\b-2}e^{h_0})\om_\vphi^n.
\end{align*}
Then rearranging these terms, we have
\begin{align*}
&E_\beta(\vphi)+\frac{1}{V}\int_M (\mathfrak h+h_0)\om_0^n\\
&=E_1(\vphi)+\frac{1-\b}{V}\int_M\log|s|_h^{2}(\om_\vphi^n-\om_0^n)-\frac{1}{V}\int_M{h_0}(\om_\vphi^n-\om_0^n).
\end{align*}
Secondly, we compute
\begin{align*}
J_{-\theta}(\vphi)&:=j_{-\theta}(\vphi)+\ul{\theta}\cdot D_{\om_0}(\vphi).
\end{align*}
Since $\theta\in C_1(X,D)=C_1(X)-(1-\b)C_1(L_D)$, we have
\begin{align*}
\ul{\theta}
:=\frac{\int_X [\theta]\Om^{n-1}}{\int_X \Om^n}=\underline S_1-(1-\b)\frac{\int_X C_1(L_D)\Om^{n-1}}{\int_X \Om^n}.
\end{align*}
Then using \eqref{h0}, we see 
\begin{align*}
j_{-\theta}(\vphi)&:=-\frac{1}{V}\int_{M}\vphi
\sum_{j=0}^{n-1}\om_0^{j}\wedge
\om_\vphi^{n-1-j}\wedge \theta\\
&=-\frac{1}{V}\int_{M}\vphi
\sum_{j=0}^{n-1}\om_0^{j}\wedge
\om_\vphi^{n-1-j}\wedge [Ric(\om_0)-(1-\beta)\Theta_D-i\p\bar\p h_0].
\end{align*}
By using the definition of $j_{-Ric(\om_0)}$ and $j_{\Theta_D}$, it is equal to
\begin{align*}
j_{-Ric(\om_0)}(\vphi)+(1-\beta)j_{\Theta_D}(\vphi) 
+\frac{1}{V}\int_{M}\vphi
\sum_{j=0}^{n-1}\om_0^{j}\wedge
\om_\vphi^{n-1-j}\wedge i\p\bar\p h_0.
\end{align*}
Further computation shows that
\begin{align*}
&\frac{1}{V}\int_{M}\vphi
\sum_{j=0}^{n-1}\om_0^{j}\wedge
\om_\vphi^{n-1-j}\wedge i\p\bar\p h_0\\
&=\frac{1}{V}\int_{M} h_0 i\p\bar\p\vphi\wedge
\sum_{j=0}^{n-1}\om_0^{j}\wedge
\om_\vphi^{n-1-j}\\
&=\frac{1}{V}\int_{M} h_0 (\om_\vphi-\om)\wedge
\sum_{j=0}^{n-1}\om_0^{j}\wedge
\om_\vphi^{n-1-j}\\
&=\frac{1}{V}\int_{M} h_0 
\sum_{j=0}^{n-1}\om_0^{j}\wedge
\om_\vphi^{n-j}-\frac{1}{V}\int_{M} h_0 
\sum_{j=0}^{n-1}\om_0^{j+1}\wedge
\om_\vphi^{n-1-j}\\
&=\frac{1}{V}\int_{M} h_0 (\om_\vphi^{n}-\om_0^n).
\end{align*}
At last, we complete the proof by adding these identities together and making use of the formula of $K$-energy \eqref{Mabuchienergy}.
\end{proof}

In the following, we obtain an equivalent formula of the log $K$-energy.
\begin{prop}\label{K and log K}On the divisor $D$, we set the corresponding volume and the normalisation functional $D$ to be
\begin{align*}
\mathrm{Vol}(D)=\int_D \Omega^{n-1},\quad D_{\om_0,D}(\vphi)=\frac{n}{V}\int_0^1\int_D \p_t\vphi \om_\vphi^{n-1}dt.
\end{align*}
Then we have
\begin{align*}
\nu_\beta(\vphi)&=\nu_1(\vphi)+(1-\beta)\cdot[D_{\om_0,D}(\vphi)
-\frac{\mathrm{Vol}(D)}{V}\cdot D_{\om_0}(\vphi)].
\end{align*}
\end{prop}
\begin{proof}
We compute that by using the Poincar\'e--Lelong equation (c.f.~\cite[Section 2.3.1]{arXiv:1803.09506})
\begin{align*}
\frac{n}{V}\int_0^1\int_D \p_t\vphi \om_\vphi^{n-1}dt
=\frac{n}{V}\int_0^1\int_X [i\p\bar\p\log|s|^2_h+\Theta_D] \p_t\vphi \om_\vphi^{n-1}dt.
\end{align*}
By \eqref{derivatives of J chi}, it is equal to
\begin{align*}
\frac{n}{V}\int_0^1\int_X i\p\bar\p\log|s|^2_h \p_t\vphi \om_\vphi^{n-1}dt+j_{\Theta_D}(\vphi) .
\end{align*}
We further compute that
\begin{align*}
\frac{n}{V}\int_0^1\int_X i\p\bar\p\log|s|^2_h \p_t\vphi \om_\vphi^{n-1}dt
&=\frac{n}{V}\int_0^1\int_X \log|s|^2_h i\p\bar\p\p_t\vphi \om_\vphi^{n-1}dt\\
&=\frac{n}{V}\int_0^1\int_M\log|s|^2_h(\omega^n_\vphi-\om^n)dt,
\end{align*}
which completes the proof.
\end{proof}

The linearity of the log $K$-energy follows directly from this formula.
\begin{thm}\label{Log K energy linear}
The log $K$-energy is linear in the cone angle $\beta$. Precisely, given $\beta_1\leq \beta_2$ and letting $\beta=(1-s)\beta_1+s\beta_2$, we have
\begin{align}
\nu_{\beta}=(1-s)\nu_{\beta_1}+s\nu_{\beta_2}.
\end{align} 
\end{thm}

\subsubsection{Uniqueness of cscK cone metrics and lower bound of the log $K$-energy}

\begin{thm}\emph{(\cite[Theorem 1.10]{MR4020314})}\label{uniqueness of cscK cone metric}
The cscK cone metric $\om_{cscK}$ with cone angle $2\pi\beta$ is unique up to automorphisms.
\end{thm}

The lower bound of the log $K$-energy is obtained from uniqueness \thmref{uniqueness of cscK cone metric} and convexity of the log $K$-energy along the cone geodesic.
\begin{thm}\emph{(\cite[Lemma 3.6]{MR4020314})}\label{Lower bound}
If a K\"ahler pair $((X,\Omega),D)$ admits a cscK cone metric $\om_{cscK}$ with cone angle $2\pi\beta$, then the log $K$-energy $\nu_\beta$ archives minimum at $\om_{cscK}$.
\end{thm}

\subsection{Automorphism group and log Futaki invariant}
We start with the following definition of automorphisms on projective variety.

\begin{defn}
	Let $(X,L)$ be a polarised smooth projective variety. We write $\mathrm{Aut}_0 (X,L)$ for the identity component of the group of biholomorphic automorphisms on $X$ whose action lifts to the total space of $L$, which is known to be a linear algebraic group. We write $\mathfrak{aut} (X,L)$ for its Lie algebra. 
\end{defn}

When we consider a polarised pair, a more appropriate definition is given as follows.

\begin{defn} \label{dfautxld}
	When $D \subset X$ be a smooth effective divisor, we write $\mathrm{Aut}_0 ((X,L);D)$ for the subgroup of $\mathrm{Aut}_0 (X,L)$ which preserves $D$; its Lie algebra $\mathfrak{aut} ((X,L);D)$ consists of holomorphic vector fields on $X$ that are tangential to $D$.
\end{defn}


We now recall the definition of the log Futaki invariant, which is defined for a holomorphic vector field and gives an obstruction to the existence of cscK cone metrics. Let $v$ be an element of $H^0 (X , T_X)$ which admits a holomorphy potential; recall that $\theta_v$ is said to be a holomorphy potential of $v$ if it satisfies
\begin{equation*}
	i_v\om = i \bar\p\theta_v .
\end{equation*}
The sign convention for the right hand side varies in the literature, but we fix the one as above. It is well-known that $ v \in H^0(X,T_X)$ admits a holomorphy potential if and only if the real part of $v$ is an element of $\mathfrak{aut} (X,L)$.

\begin{defn}
	The \textbf{Futaki invariant} of $v \in H^0 (X , T_X)$ with the holomorphy potential $\theta_v$ is defined by
	\begin{equation*}
		\mathrm{Fut} (v) := -  \int_X \theta_v (S(\omega) - \underline{S}_1) \frac{\omega^n}{n!},
	\end{equation*}
	where $\underline{S}_1$ is as defined in (\ref{eqavlgsc}). For a smooth effective divisor $D \subset X$, the \textbf{log Futaki invariant} with cone angle $2 \pi \beta$ is defined by
	\begin{align}\label{Log DF invariant}
\mathrm{Fut}_{D , \beta}(v) := \frac{1}{2 \pi} \mathrm{Fut}(v )+(1-\beta) \left( \int_{D}\theta_v \frac{\omega^{n-1}}{(n-1)!} -\frac{\mathrm{Vol}(D)}{V}\int_X\theta_v \frac{\omega^n}{n!} \right).
\end{align}
\end{defn}

While the Futaki invariant and the log Futaki invariant are both defined with respect to a K\"ahler form $\omega$ on $X$, it only depends on its cohomology class $[\omega]$, as proved in \cite{MR718940,MR2975584}.

Comparing with Proposition \ref{K and log K}, we see that
the log Futaki invariant is the gradient of the log $K$-energy.

We finish off this section by providing some auxiliary results on the log Futaki invariant and the automorphism group which we use later.

\begin{prop} \label{ppcnaglfut}
Suppose that there exists $v \in \mathfrak{aut} ((X,L);D)$ such that
\begin{equation} \label{dennonzero}
	\int_{D}\theta_v \frac{\omega^{n-1}}{(n-1)!} -\frac{\mathrm{Vol}(D)}{V}\int_X\theta_v \frac{\omega^n}{n!}  \neq 0.
\end{equation}
Then the cone angle $2\pi\beta$ of the cscK cone metric, if exists, is given by
\begin{align}\label{beta D}
\beta = 1+ \mathrm{Fut} (v)\left( \int_{D}\theta_v \frac{\omega^{n-1}}{(n-1)!} -\frac{\mathrm{Vol}(D)}{V}\int_X\theta_v \frac{\omega^n}{n!} \right)^{-1}.
\end{align}
\end{prop}
\begin{proof}
From coercivity of the log $K$-energy, we see that it is bounded below by \thmref{Lower bound}.
The derivative of the log $K$-energy $\nu_\beta$ is log Futaki invariant $\mathrm{Fut}_{D,\beta}$. So $\mathrm{Fut}_{D,\beta}$ must vanish, by noting $\mathrm{Fut}_{D,\beta} (-v) = - \mathrm{Fut}_{D,\beta} (v)$. Then by \eqref{Log DF invariant}, the cone angle $2\pi\beta$ must satisfy the identity \eqref{beta D}.
\end{proof}

The condition (\ref{dennonzero}) is not vacuous, and satisfied e.g.~for the case considered in \cite[Theorem 1.7]{MR3978320}. The above lemma shows that, when there is a nontrivial holomorphic vector field $v$ that preserves $D$ the cone angle is uniquely determined by the vanishing of the log Futaki invariant (as long as $v$ is ``generic'' in the sense that it satisfies (\ref{dennonzero})).

On the other hand, such cases happen only if we choose the divisor $D$ to be non-generic, to the extent that $D$ is preserved by a nontrivial holomorphic vector field at all. Indeed we expect $\mathrm{Aut}_0((X,L);D)$ to be trivial for a generically chosen divisor $D$, even when $\mathrm{Aut}_0 (X,L)$ is not. We partially confirm this expectation by giving a simple sufficient condition for the triviality of $\mathrm{Aut}_0((X,L);D)$ in the following proposition.

\begin{proposition} \label{ppautmori}
	Let $A$ be an ample line bundle on $X$. Then there exists $m_1 \in \mathbb{N}$ depending only on $X$ and $A$ such that for all $m \ge m_1$ we have $\mathrm{Aut}_0((X,L);D_m) = 0$ for a generic member $D_m$ of the linear system $|mA|$.
\end{proposition}

The proposition above can be regarded as a partial generalisation of the result by Song--Wang \cite[Theorem 2.8]{MR3470713} for Fano manifolds (see also Berman \cite[\S 1.6]{MR3107540}), and the proof below follows the same strategy as theirs; note on the other hand that the aforementioned papers \cite{MR3107540,MR3470713} provide a sharper estimate for the $m_1$ above for the Fano case.

\begin{proof}
	We first prove that there exists $m_0 \in \mathbb{N}$ depending only on $X$ and $A$ such that a generic member of the linear system $|mA|$ admits no nontrivial holomorphic vector fields for all $m \ge m_0$. Mori's cone theorem \cite{MR662120} implies that $K_X + (n+2) A$ is an ample divisor on $X$ \cite[Theorem 1.5.33 and Example 1.5.35]{MR2095471}. Pick $m_0 \ge n+2$ so that $|m_0 A|$ is basepoint free, noting that how large $m_0$ should be depends only on $X$ and $A$. Then, for $m \ge m_0$, a generic member $D_m \in |mA|$ is smooth by Bertini. The adjunction formula gives $K_{D_m} = (K_X + D_m) |_{D_m} = (K_X + mA) |_{D_m}$, which is ample for all $m \ge m_0$ since it is a restriction of an ample divisor to $D_m$. Thus $D_m$ admits no nontrivial holomorphic vector fields by \cite[Chapter III, Theorem 2.4]{MR1336823}.
	
	We then prove that there exists $m_1 \in \mathbb{N}$ depending only on $X$ and $A$ such that for all $m \ge m_1$ and for a generic $D_m \in |mA|$ there exist no nontrivial holomorphic vector fields on $X$ that vanish on $D_m$. It is necessary and sufficient to prove that there exists $m_1 \in \mathbb{N}$ such that
	\begin{equation*}
		H^0(X , T_X \otimes \mathcal{O}_X (-D_m) ) = H^0(X , T_X \otimes A^{\otimes (-m)} ) = 0
	\end{equation*}
	for the holomorphic tangent sheaf $T_X$ of $X$ and for all $m \ge m_1$; this holds indeed since
	\begin{equation*}
		H^0(X , T_X \otimes A^{\otimes (-m)} ) = H^n (X , K_X  \otimes T^*_X \otimes A^{\otimes m} )^{\vee} = 0
	\end{equation*}
	for all large enough $m$ by the Serre duality and the Serre vanishing, where we note how large $m$ should be depends only on $X$ and $A$ \cite[Theorem 1.2.6]{MR2095471}.
	
	Combined with the previous claim and replacing $m_1$ by $\max \{m_0 , m_1 \}$ if necessary, we then find that, for all $m \ge m_1$ and a generic $D_m \in |mA|$, there exist no nontrivial holomorphic vector fields on $X$ that preserves $D_m$. Hence we get $\mathfrak{aut} ((X,L);D_m) = 0$, which obviously implies $\mathrm{Aut}_0((X,L);D_m) = 0$.
\end{proof}


\section{Review of various log stabilities} \label{scrvlst}

We review various stability notions for a polarised pair $((X,L);D)$. Sections \ref{Test configurations and log K stability} and \ref{Uniform log K stability} concern variants of the log $K$-stability, which is formulated in an algebro-geometric language. Section \ref{scrvcscsklp}, on the other hand, discusses the log properness and the log geodesic stability which is essentially analytic and play an important role in this paper.

\subsection{Test configurations and log $K$-stability}\label{Test configurations and log K stability}
We recall the basic definitions and facts concerning the log $K$-stability and its variants. The reference is \cite{MR3669511,MR3403730}. We start by recalling the following.


\begin{defn}[Test configuration] \label{dflgtc}
A \textbf{test configuration} $(\mathcal{X} , \mathcal{L})$ for a polarised K\"ahler manifold $(X,L)$ is a scheme $\mathcal{X}$ with a flat projective morphism $\pi : \mathcal{X} \to \mathbb{C}$ such that
\begin{itemize}
	\item $\mathbb{C}^*$ acts on $\mathcal{X}$ in such a way that $\pi$ is $\mathbb{C}^*$-equivariant, with a linearisation of the $\mathbb{C}^*$-action on $\mathcal{L}$,
	\item $\pi^{-1} (1) = (X,L)$.
\end{itemize}
The fibre $\pi^{-1} (0)$ over $0 \in \mathbb{C}$ is called the \textbf{central fibre} and written also as $\mathcal{X}_0$.

A \textbf{log test configuration} $((\mathcal{X}, \mathcal{L}); \mathcal{D})$ for a polarised pair $((X,L);D)$ is a test configuration $(\mathcal{X} , \mathcal{L})$ for $(X,L)$ together with a subscheme $\mathcal{D} \subset \mathcal{X}$ obtained by complementing the $\mathbb{C}^*$-orbit of $D$ in $\mathcal{X} \setminus \pi^{-1} (0)$ with the flat limit over $0 \in \mathbb{C}$.
\end{defn}

We say that a log test configuration $((\mathcal{X}, \mathcal{L}); \mathcal{D})$ is \textbf{normal} if $\mathcal{X}$ is a normal variety. We also say that $((\mathcal{X}, \mathcal{L}); \mathcal{D})$ is \textbf{product} if $\mathcal{X}$ is isomorphic to $X \times \mathbb{C}$ and $\mathcal{D}$ is isomorphic to $D \times \mathbb{C}$. Note that a product log test configuration corresponds one-to-one with a holomorphic vector field $v$ admitting a holomorphy potential which is tangential to the divisor $D \subset X$ (i.e.~the real part of $v$ is in $\mathfrak{aut}((X,L);D)$); $v$ preserving $D$ is important, since otherwise we have a log test configuration with $\mathcal{X} \cong X \times \mathbb{C}$ but $\mathcal{D} \not\cong D \times \mathbb{C}$. Finally, a log test configuration is said to be \textbf{trivial} if $\mathcal{X}$ (resp.~$\mathcal{D}$) is $\mathbb{C}^*$-equivariantly isomorphic to $X \times \mathbb{C}$ (resp.~$D \times \mathbb{C}$), meaning that the $\mathbb{C}^*$-action on $\mathcal{X} \cong X \times \mathbb{C}$ (resp.~$\mathcal{D} \cong D \times \mathbb{C}$) induces a trivial $\mathbb{C}^*$-action on the first factor $X$ (resp.~$D$).

We also recall that we can compactify a test configuration to form a family over $\mathbb{P}^1$ (see also \cite[section 2.2]{MR3669511}).

\begin{defn} \label{defcptcx}
	Let $(\mathcal{X} , \mathcal{L})$ be a test configuration for $(X,L)$. The \textbf{compactification} $\bar{\mathcal{X}}$ of $\mathcal{X}$ is defined by gluing together $\mathcal{X}$ and $X \times (\mathbb{P}^1 \setminus \{ 0 \})$ along their respective open subsets $\mathcal{X} \setminus \mathcal{X}_0$ and $X \times (\mathbb{C} \setminus \{ 0 \} )$, identified by the canonical $\mathbb{C}^*$-equivariant isomorphism $\mathcal{X} \setminus \mathcal{X}_0 \simeq X \times (\mathbb{C} \setminus \{ 0 \} )$. The line bundle $\bar{\mathcal{L}}$ over $\bar{\mathcal{X}}$ can be defined as the natural one via the procedure for the compactification above.
\end{defn}

The log Donaldson--Futaki invariant is an algebro-geometric invariant associated to a log test configuration, defined as follows. Let $A_k$ be the infinitesimal generator for the $\mathbb C^\ast$ acting on the vector space $H^0(\mathcal X_0, \mathcal{L}^{\otimes k}|_{\mathcal{X}_0})^{\vee}$, which is the dual of the set of holomorphic sections of $\mathcal{L}^{\otimes k}|_{\mathcal{X}_0}$. We denote the dimension of $H^0(\mathcal X_0, \mathcal L_0^{\otimes k})$ by $d_k$ and the trace of $A_k$ by $w_k$.

According to the Riemann--Roch theorem and its equivariant version, $d_k$ and $w_k$ are given by the polynomials
\begin{align}
d_k &=a_0 k^n+a_1 k^{n-1} +O(k^{n-2}), \label{eqaedk} \\
w_k &=b_0k^{n+1}+b_1 k^{n} +O(k^{n-1}), \label{eqaewk}
\end{align}
where $a_0, a_1, b_0, b_1, \dots$ are rational numbers, for sufficiently large $k$. Restricted to $\mathcal{D}$, the corresponding dimension $\tilde d_k:=\dim H^0(\mathcal D_0, \mathcal{L}^{\otimes k}|_{\mathcal{D}_0})$ and the weight $\tilde w_k$ of $\mathbb{C}^* \curvearrowright H^0(\mathcal D_0, \mathcal{L}^{\otimes k}|_{\mathcal{D}_0})^{\vee}$ satisfy similar formulae
\begin{align}
\tilde d_k &=\tilde  a_0 k^n+\tilde a_1 k^{n-1} +O(k^{n-2}), \label{eqaedkt} \\ 
\tilde w_k &=\tilde b_0k^{n+1}+\tilde b_1 k^{n} +O(k^{n-1}). \label{eqaewkt}
\end{align}

\begin{defn}
	The \textbf{log Donaldson--Futaki invariant} is defined by
	\begin{align}
	DF (\mathcal{X}, \mathcal{D} , \mathcal{L} , \beta ) =\frac{2(a_1b_0-a_0b_1)}{a_0}+(1-\b)\frac{a_0\tilde b_0-\tilde a_0 b_0}{a_0}.
	\end{align}
\end{defn}

\begin{rem}
	The sign above is different from the one given e.g.~in \cite[Definition 6.7]{MR3186384}, but this is because we considered the action on the dual vector space. The sign changes when we consider the action on $H^0(\mathcal X_0, \mathcal{L}^{\otimes k}|_{\mathcal{X}_0})$, rather than its dual, which is more adapted to the differential-geometric formula as in \cite[Proposition 7.15]{MR3186384}.
\end{rem}

Similarly to \cite{MR2192937}, the log Futaki invariant can be regarded as the limit of the ``log'' Chow weight.

We are now ready to define the log $K$-stability.

\begin{defn}[Log $K$-stability]
	The polarised pair $((X,L);D)$ is said to be \textbf{log $K$-semistable} with angle $2 \pi \beta$ if the log Donaldson--Futaki invariant satisfies $DF (\mathcal{X}, \mathcal{D} , \mathcal{L} , \beta ) \geq 0$ for any normal log test configuration $((\mathcal{X}, \mathcal{L}); \mathcal{D})$.
	
We say $((X,L);D)$ is \textbf{log $K$-polystable} with angle $2 \pi \beta$ if $((X,L);D)$ is $K$-semistable and $DF (\mathcal{X}, \mathcal{D} , \mathcal{L} , \beta ) =0 $ if and only if $((\mathcal{X}, \mathcal{L}); \mathcal{D})$ is a product log test configuration.

Finally, $((X,L);D)$ is said to be \textbf{log $K$-stable} with angle $2 \pi \beta$ if $((X,L);D)$ is $K$-semistable and $DF (\mathcal{X}, \mathcal{D} , \mathcal{L} , \beta ) =0 $ if and only if $((\mathcal{X}, \mathcal{L}); \mathcal{D})$ is a trivial log test configuration.
\end{defn}

A subtle point in the definition of the log $K$-semistability above is that we decreed $DF (\mathcal{X}, \mathcal{D} , \mathcal{L} , \beta ) \geq 0$ for \textit{normal} log test configurations. It is known that the usual Donaldson--Futaki invariant decreases when we take the normalisation of $\mathcal{X}$ \cite[Proposition 3.15]{MR3669511}, and hence it suffices to test the positivity of $DF ( \mathcal{X}, \mathcal{L} )$ for all normal test configurations. On the other hand, for the log case, how the extra term behaves under the normalisation does not seem completely trivial. While it may happen that a similar result holds for the log case, in this paper we just \textit{decree} that we only check the Donaldson--Futaki invariant for normal test configurations.

Note moreover that the condition $DF (\mathcal{X}, \mathcal{D} , \mathcal{L} , \beta ) =0 $ for product log test configurations is equivalent to the vanishing of the log Futaki invariant \eqref{Log DF invariant}, i.e. $\mathrm{Fut}_{D , \beta} (v) = 0$, for any holomorphic vector field $v$ admitting a holomorphic potential and preserving $D$.

We also need the following quantity which serves as the ``norm'' of test configurations. This was introduced by Boucksom--Hisamoto--Jonsson \cite{MR3669511}. We follow the exposition in \cite[Definition 2.3]{MR3956698}.

\begin{defn} \label{defjna}
	Let $(\mathcal{X}, \mathcal{L})$ be a normal test configuration for $(X,L)$ and
	\begin{displaymath}
		\xymatrix{
		& \bar{\mathcal{Z}} \ar[dr]^{\Theta} \ar[dl]_{\Pi} & \\
		X \times \mathbb{P}^1 &  & \bar{\mathcal{X}}
		}
	\end{displaymath}
	be the normalisation of the graph of the birational map $X \times \mathbb{P}^1 \dashrightarrow \bar{\mathcal{X}}$. The \textbf{non-Archimedean $J$-functional} of $(\mathcal{X}, \mathcal{L})$ is defined by
	\begin{equation*}
		J^{\mathrm{NA}} (\mathcal{X}, \mathcal{L}) := \frac{\Pi^* \mathrm{pr}_1^* L^{\cdot n} \cdot \Theta^* \bar{\mathcal{L}}}{\int_X C_1(L)^n} - \frac{\bar{\mathcal{L}}^{\cdot n+1}}{(n+1) \int_X C_1(L)^n},
	\end{equation*} 
	where $ \mathrm{pr}_1 : X \times \mathbb{P}^1 \to X$ is the natural projection and the numerator of the first (resp.~second) term is the intersection product on $\bar{\mathcal{Z}}$ (resp.~$\bar{\mathcal{X}}$).
\end{defn}

The non-Archimedean $J$-functional has some nontrivial and interesting properties that are not immediately obvious from the definition. It is always a nonnegative rational number and zero if and only if $(\mathcal{X}, \mathcal{L})$ is (normal and) trivial \cite[Theorem 7.9]{MR3669511}.

It is also well-known \cite[Proposition 7.8 and Remark 7.12]{MR3669511} that $J^{\mathrm{NA}}$ is Lipschitz equivalent to the minimum norm introduced by Dervan \cite[Definition 4.5]{MR3564626}, which is defined as
\begin{equation} \label{dfminnorm}
	\Vert (\mathcal{X} , \mathcal{L}) \Vert_m := \frac{\tilde{b}_0a_0 - b_0 \tilde{a}_0}{a_0},
\end{equation}
where $a_0$, $b_0$ are as in (\ref{eqaedk}) and (\ref{eqaewk}), and $\tilde{a}_0$, $\tilde{b}_0$ are as in (\ref{eqaedkt}) and (\ref{eqaewkt}) defined for a \textit{generic} member of the linear system $|L|$.

We also note that $J^{\mathrm{NA}}$ agrees with the asymptotic slope of the functional $J^A_{\om_0}$ (\ref{defaubjfc}) along a psh ray corresponding to a test configuration; this fact will be used later in Section \ref{sclogkpolyst}, but for the moment we merely refer the reader to \cite[Section 3]{MR3985614} for the precise statement and more details.


\subsection{$G$-uniform log $K$-stability}\label{Uniform log K stability}

We start by recalling the notion of the uniform log $K$-stability, which can be defined following Boucksom--Hisamoto--Jonsson \cite{MR3669511} and Dervan \cite{MR3564626}.

\begin{defn}[Uniform log $K$-stability]
A polarised pair $((X,L);D)$ is \textbf{uniformly log $K$-stable} with angle $2 \pi \beta$ if there exists a positive constant $\epsilon >0$ such that 
\begin{equation*}
	DF (\mathcal{X}, \mathcal{D} , \mathcal{L} , \beta ) \geq \epsilon J^{\mathrm{NA}}(\mathcal{X}, \mathcal{L})
\end{equation*}
for any normal log test configuration $((\mathcal{X}, \mathcal{L}); \mathcal{D})$ for $((X,L);D)$.
\end{defn}

Dervan \cite[Definition 4.5]{MR3564626} used the minimum norm in the definition of the uniform $K$-stability, but it is well-known \cite[Remark 8.3]{MR3669511} that it is equivalent to the formulation above which uses $J^{\mathrm{NA}}$.

\begin{rem}
The above stability condition is closely related to the \textit{uniformly twisted $K$-stability} introduced by Dervan \cite[Definition 2.7]{MR3564626}, in the sense that it is defined as the uniform log $K$-stability when the divisor $D$ is a general member of a fixed linear system \cite[Section 2.1]{MR3564626}. We shall return to this point when we later discuss Theorem \ref{thmalphainv}.	
\end{rem}


There is a version of the uniform $K$-stability when we have a nontrivial automorphism group $G$; see \cite{arXiv:1610.07998} and \cite[Section 3]{arXiv:1907.09399} for more details. We assume in the following definition that $G := \mathrm{Aut}_0 ((X,L);D)$ is reductive, but it makes sense when $G$ is merely a reductive subgroup of $\mathrm{Aut}_0 ((X,L);D)$. We write $T$ for the complex torus which is the identity component of the centre of $G$. We define the cocharacter lattice $N_{\mathbb{Z}}$ of $T$ as
\begin{equation*}
	N_{\mathbb{Z}} := \mathrm{Hom} (\mathbb{C}^* , T),
\end{equation*}
and write $N_{\mathbb{R}} := N_{\mathbb{Z}} \otimes_{\mathbb{Z}} \mathbb{R}$. We recall the following definition from \cite[Definition 3.1]{arXiv:1907.09399}.

\begin{defn}[$G$-equivariant test configuration]
	We say that a log test configuration $((\mathcal{X}, \mathcal{L}); \mathcal{D})$ is \textbf{$G$-equivariant} if $G$ acts on $\mathcal{X}$ in such a way that it commutes with the $\mathbb{C}^*$-action of $((\mathcal{X}, \mathcal{L}); \mathcal{D})$, preserves $\mathcal{D}$, admits a linearisation to $\mathcal{L}$, and agrees with the action $G \curvearrowright ((X,L);D)$ when restricted to $\pi^{-1} (1)$.
\end{defn}

A $G$-equivariant test configuration can be twisted by a $T$-action, as in \cite[Definition 3.2]{arXiv:1907.09399}.

\begin{defn}
	Let $((\mathcal{X}, \mathcal{L}); \mathcal{D})$ be a $G$-equivariant log test configuration. For $\xi \in N_{\mathbb{Z}}$, the \textbf{$\xi$-twist of $((\mathcal{X}, \mathcal{L}); \mathcal{D})$}, written $((\mathcal{X}_{\xi}, \mathcal{L}_{\xi}); \mathcal{D}_{\xi})$, is a test configuration with the total space $((\mathcal{X}, \mathcal{L}); \mathcal{D})$ but with the $\mathbb{C}^*$-action generated by $\eta + \xi$, where $\eta$ is the generator of the $\mathbb{C}^*$-action defining $((\mathcal{X}, \mathcal{L}); \mathcal{D})$.
\end{defn}

We also allow $\xi$ to be an element of $N_{\mathbb{R}}$; in this case the resulting $\xi$-twist $((\mathcal{X}_{\xi}, \mathcal{L}_{\xi}); \mathcal{D}_{\xi})$ is no longer a test configuration strictly speaking, since the $\mathbb{C}^*$-action is no longer algebraic, and hence it is called an $\mathbb{R}$-test configuration in \cite{arXiv:1610.07998} or \cite{arXiv:1907.09399}. We note that its compactification $(\bar{\mathcal{X}}_{\xi}, \bar{\mathcal{L}}_{\xi})$ (as in Definition \ref{defcptcx}) makes sense as a complete analytic space which is in general not isomorphic to $(\bar{\mathcal{X}}, \bar{\mathcal{L}})$; in particular, $J^{\mathrm{NA}} (\mathcal{X}_{\xi}, \mathcal{L}_{\xi})$ is well-defined by considering a bimeromorphic map $X \times \mathbb{P}^1 \dashrightarrow \bar{\mathcal{X}}_{\xi}$ in Definition \ref{defjna}, and depends nontrivially on $\xi$. With this understood, we define the following.

\begin{defn}[Reduced $J$-norm]
	Let $(\mathcal{X}, \mathcal{L})$ be a $T$-equivariant test configuration. The \textbf{reduced $J$-norm} of $(\mathcal{X}, \mathcal{L})$ is defined by
	\begin{equation*}
		J_T^{\mathrm{NA}} (\mathcal{X}, \mathcal{L}) := \inf_{\xi \in N_{\mathbb{R}}} J^{\mathrm{NA}} (\mathcal{X}_{\xi}, \mathcal{L}_{\xi}).
	\end{equation*}
\end{defn}

Note that a $G$-equivariant test configuration is necessarily $T$-equivariant. We finally arrive at the definition of the $G$-uniform log $K$-stability.

\begin{defn}[$G$-uniform log $K$-stability]
	A polarised pair $((X,L);D)$ is said to be \textbf{$G$-uniformly log $K$-stable} if there exists $\epsilon >0$ such that for any $G$-equivariant normal log test configuration $((\mathcal{X}, \mathcal{L}); \mathcal{D})$ we have
	\begin{equation*}
		DF (\mathcal{X}, \mathcal{D} , \mathcal{L} , \beta ) \ge \epsilon J_T^{\mathrm{NA}} (\mathcal{X}, \mathcal{L}).
	\end{equation*}
\end{defn}



\subsection{Review of cscK cone metrics and log properness} \label{scrvcscsklp}

We review the results proved in \cite{arXiv:1803.09506}, which many results in this paper rely on. Note that the results below hold for a general compact K\"ahler manifold $X$ with the K\"ahler class $\Omega$.

\subsubsection{Log properness}
We set $G:=\mathrm{Aut}_0((X,\Omega);D)$.
The $G$-orbit of a K\"ahler cone potential $\vphi\in\mathcal H_\b$ is defined to be
\begin{align*}
\mathcal O_\vphi=\{\tilde{\vphi}\vert\om_{\tilde{\vphi}}=\sigma^\ast\om_ \vphi, \forall \sigma\in G\}.
\end{align*}
The $d_{1,G}$-distance between two K\"ahler cone potentials $\vphi_1,\vphi_2\in \mathcal H_\b$ is defined to be the infimum of the $d_1$-distance between the corresponding orbits $\mathcal O_{\vphi_1}$ and $\mathcal O_{\vphi_2}$. Since K\"ahler cone potential could be approximated by smooth K\"ahler potentials, it is sufficient to consider properness and coercivity in the space $\mathcal H$ of smooth K\"ahler potentials associated to $\Omega$. 
\begin{defn}
	The log $K$-energy $\nu_\beta$ is said to be \textbf{proper}, if for any sequence $\{\vphi_i\}\subset \mathcal H$, 
	\begin{align*}
	\lim_{i\rightarrow\infty}d_{1,G}(0,\vphi_i)=\infty \implies \lim_{i\rightarrow\infty}\nu_\beta(\vphi_i)=\infty.	
	\end{align*}
\end{defn}
\begin{defn}
The log $K$-energy $\nu_\beta$ is said to be \textbf{$d_{1,G}$-coercive}, if there exists positive constants $A$ and $B$ such that
\begin{equation} \label{d1gcoercive}
	\nu_\beta(\vphi)\geq A \cdot d_{1,G}(\vphi,0)-B
\end{equation}
for all $\vphi\in \mathcal H$.
\end{defn}
Clearly, coercivity implies properness. It is a quantitative version of properness.
\begin{thm}\emph{(Log properness theorem \cite[Theorem 1.2, Theorem 7.4]{arXiv:1803.09506})}\label{cone properness theorem}
The K\"ahler pair $((X,\Omega);D)$ admitting a cscK cone metric with cone angle $2\pi\beta$ is equivalent to the properness of the log $K$-energy $\nu_\beta$, which is also equivalent to the $d_{1,G}$-coercivity of $\nu_\beta$.
\end{thm}
In particular, when the automorphism group $\mathrm{Aut}_0((X,\Omega);D)$ is trivial, we have
\begin{thm}\emph{(\cite[Theorem 1.2, Theorem 7.2]{arXiv:1803.09506})}\label{cone properness conjecture} Assume that the automorphism group is discrete.
The K\"ahler pair $((X,\Omega);D)$ admitting a cscK cone metric with cone angle $2\pi\beta$ is equivalent to the properness of the log $K$-energy $\nu_\beta$, and also to the $d_{1}$-coercivity of $\nu_\beta$.
\end{thm}
These existence theorems lead to several applications relating to various notions of log $K$-stabilities, as we will see in the rest of this article.


\subsubsection{Log geodesic stability}\label{Log geodesic stability}
\begin{defn}Suppose $\rho(t):[0,\infty)\rightarrow \mathcal E^1_0$ is a $d_1$-geodesic ray. The \textbf{$\mathfrak  F$-invariant} corresponding to the log $K$-energy along the $d_1$-geodesic ray is defined to be the slope,
\begin{align*}
\mathfrak F(\rho(t))=\lim_{t\rightarrow \infty} \frac{\nu_\beta(\rho(t))}{t}.
\end{align*}
\end{defn}

A ray $\rho(t)\in  \mathcal E^1_0$ starting at $\vphi_0$ is called \textit{holomorphic}, if it is generated by a one-parameter holomorphic action $\sigma(t)\in \mathrm{Aut}_0(X;D)$, i.e. $\om_{\rho(t)}=\sigma(t)^\ast\om_{\vphi_0}$. 
Two rays $\rho_1(t), \rho_2(t)$ are \textit{parallel} if they have uniformly bounded $d_1$-distance. 
A ray is \textit{trivial}, if it is parallel to a holomorphic ray.

\begin{defn}[Log geodesic stability]\label{log geodesic stability}
Let $\rho(t):[0,\infty)\rightarrow \mathcal E^1_0$ be any unit speed $d_1$-geodesic starting at $\vphi_0\in \mathcal E^1_0$.
\begin{itemize}
\item
The point $\vphi_0$ is \textbf{log geodesic semi-stable}, if $\mathfrak F(\rho(t))\geq 0$.
\item
The point $\vphi_0$ is \textbf{log geodesic stable}, if it is geodesic semi-stable and the equality holds when $\rho(t)$ is trivial.
\end{itemize}
A K\"ahler class $\Om$ is geodesic stable (resp.~geodesic semi-stable), if every $\vphi_0\in \mathcal E^1_0$ is geodesic stable (resp.~geodesic semi-stable).
\end{defn}
Actually, in the definition above, the condition that $\rho(t)$ is trivial could be replaced by an equivalent condition, i.e. $\rho(t)$ is holomorphic \cite{MR4095424}.
It is shown in \cite{arXiv:1803.09506} that
\begin{thm}\label{geodesic stability}\emph{(\cite[Theorem 1.8]{arXiv:1803.09506})}
$(M,\Om)$ admits a cscK cone metric if and only if it is log geodesic stable.
\end{thm}





\section{Results on stability} \label{scrsostab}

\subsection{Statement of the results}

In this section, we will prove that
\begin{thm}\label{log K}
	Suppose that $((X,L);D)$ admits a cscK cone metric of angle $2 \pi \beta$. Then the following hold.
	\begin{enumerate}
	\item \label{log K semistability} $((X,L);D)$ is log $K$-semistable with angle $2 \pi \beta$;
	\item \label{log K polystability} $((X,L);D)$ is log $K$-polystable with angle $2 \pi \beta$;
	\item \label{uniform log K-stability} $((X,L);D)$ is $G$-uniformly log $K$-stable with angle $2 \pi \beta$ for $G = \mathrm{Aut}_0 ((X,L);D)$;
	\item \label{log K stability} $((X,L);D)$ is uniformly log $K$-stable with angle $2 \pi \beta$ if the automorphism group $\mathrm{Aut}_0 ((X,L);D)$ is further assumed to be trivial.
	\end{enumerate}
\end{thm}

\begin{rem}
Noting that uniform log $K$-stability implies log $K$-stability \cite{MR3564626}, we find that the item  \ref{log K stability} above implies the log $K$-stability of $((X,L);D)$ with a cscK cone metric. Note also that the log $K$-polystability implies the equivariant log $K$-polystability.
\end{rem}

Since the log $K$-stability implies the twisted $K$-stability \cite[Theorem 1.10]{MR3564626}, we have
\begin{thm}\label{twistedK stability}
	Suppose that $((X,L);D)$ admits a cscK cone metric of angle $2 \pi \beta$. Then $((X,L);\frac{1-\beta}{2}L_D)$ is twisted $K$-stable.
\end{thm}

\begin{rem}
	The slope stability introduced in Ross--Thomas \cite{MR2819757} was generalised to the log setting, see \cite{MR3248054} for the discussion for K\"ahler--Einstein cone metrics and also Section \ref{scavsclki}. The log $K$-semistability with angle $2\pi\beta$ implies the log slope semistability with angle $2\pi\beta$.
\end{rem}

We also prove a certain sufficient condition for the log $K$-instability in terms of the comparison of the cone angle and the average value of the scalar curvature of the divisor; see Section \ref{scavsclki}, particularly Theorem \ref{destabilizing test configuration}, for the more precise statement and the details.


\subsection{Proof of \thmref{log K}}\label{sclogkpolyst}

We first recall a foundational results by Boucksom--Hisamoto--Jonsson \cite{MR3985614} that we use in the proof.

\begin{thm} \emph{(Boucksom--Hisamoto--Jonsson \cite[Theorem 4.2]{MR3985614})} \label{thmbhj}
	Suppose that $(\phi_t)_{t \ge 0}$ is a smooth subgeodesic ray in $\mathcal{H} (L)$ which admits a non-Archimedean limit $\phi^{\mathrm{NA}} \in \mathcal{H}^{\mathrm{NA}} (L)$. Then
	\begin{equation*}
		\lim_{t \to + \infty} \frac{\nu_{\beta} (\phi_t)}{t} = \frac{1}{V} \left( K^{\mathrm{log}}_{( \bar{\mathcal{X}} , (1 - \beta )\bar{\mathcal{D}}) / \mathbb{P}^1} \cdot \bar{\mathcal{L}}^n + \frac{\underline{S}_{\beta}}{n+1} \bar{\mathcal{L}}^{n+1} \right)
	\end{equation*}
	for a normal representative $(\mathcal{X} , \mathcal{L})$ of $\phi^{\mathrm{NA}}$,
	where $(\bar{\mathcal{X}} , \bar{\mathcal{L}})$ is the compactification of $(\mathcal{X} , \mathcal{L})$ as in Definition \ref{defcptcx}, $\underline{S}_{\beta}$ is as defined in (\ref{eqavlgsc}), and
	\begin{align*}
		K^{\mathrm{log}}_{( \bar{\mathcal{X}} , (1 - \beta ) \bar{\mathcal{D}}) / \mathbb{P}^1} &:= K^{\mathrm{log}}_{\bar{\mathcal{X}} / \mathbb{P}^1} + (1 - \beta ) \mathcal{D} \\
		&=K_{\bar{\mathcal{X}} / \mathbb{P}^1} + \mathcal{X}_{0, \mathrm{red}} - \mathcal{X}_0 + (1 - \beta ) \mathcal{D} ,
	\end{align*}
	where $\mathcal{D}$ is defined as in Definition \ref{dflgtc} in terms of the divisor $D \subset X$.
\end{thm}

Note that $K^{\mathrm{log}}_{( \bar{\mathcal{X}} , (1 - \beta ) \bar{\mathcal{D}}) / \mathbb{P}^1}$ makes sense as a Weil divisor in $\mathcal{X}$ since it is normal.

We decide not to explain in this paper what a ``non-Archimedean limit'' or $\mathcal{H}^{\mathrm{NA}} (L)$ is, since it is rather technical. The only fact that we need is that a test configuration naturally defines a smooth subgeodesic ray $(\phi_t)_{t \ge 0}$ in $\mathcal{H} (L)$ for which the non-Archimedean limit $\phi^{\mathrm{NA}}$ exists and \thmref{thmbhj} holds. The reader is referred to \cite{MR3985614,MR3669511} for more details.



\thmref{log K} can be obtained essentially as a consequence of the above result of Boucksom--Hisamoto--Jonsson \cite{MR3985614}, the log properness result (Theorems \ref{cone properness theorem} and \ref{cone properness conjecture}), and the log geodesic stability result (Theorem \ref{geodesic stability}) of the third author \cite{arXiv:1803.09506}.


\subsubsection{Proof of log $K$-semistability} \label{scpflkss}

We start by proving (\ref{log K semistability}) of \thmref{log K}. As we shall see below, for this purpose we only need to assume that $\nu_{\beta}$ is bounded below, which is weaker than the existence of conic cscK metrics.

\begin{proof}[Proof of (\ref{log K semistability}) in \thmref{log K}, log $K$-semistability]
	By \thmref{geodesic stability} we know that $\nu_{\beta}$ is geodesically stable for all finite energy geodesics in $(\mathcal{E}^1(L) , d_1)$, i.e.
	\begin{equation*}
		\lim_{t \to + \infty} \frac{\nu_{\beta} (\psi_t)}{t} \ge 0
	\end{equation*}
	for any finite energy geodesic $( \psi_t )_{t \ge 0}$; note that the above inequality holds if $\nu_{\beta}$ is merely bounded below. Recall that, for any smooth subgeodesic ray $(\phi_t)_{t \ge 0}$ which extends to the total space of a normal test configuration $(\mathcal{X} , \mathcal{L})$, we can construct a finite energy geodesic $( \psi_t )_{t \ge 0}$ such that
	\begin{equation} \label{eqsbgfeg}
		\lim_{t \to + \infty} \frac{\nu_{\beta} (\phi_t)}{t} = \lim_{t \to + \infty} \frac{\nu_{\beta} (\psi_t)}{t},
	\end{equation}
	which follows from the result proved by Berman \cite{MR3461370}, Chen--Tang \cite{MR2521647}, and Phong--Sturm \cite{MR224263,MR2377252,MR2661562}, by also noting that the uniform convergence as $t \to + \infty$ proved in these papers and \cite[Lemma 3.9]{MR3985614} give (\ref{eqsbgfeg}); see in particular \cite[Proposition 2.7 and the proof of Proposition 3.6]{MR3461370} for more details, and \cite[Theorem 6.6]{arXiv:1509.04561} for the generalisation to $\mathcal{E}^{1,\mathrm{NA}}$.

	Suppose that we have a normal test configuration $(\mathcal{X} , \mathcal{L})$. Then we can take a smooth subgeodesic $(\phi_t)_{t \ge 0}$ in $\mathcal{H} (L)$ which admits a non-Archimedean limit $\phi^{\mathrm{NA}} \in \mathcal{H}^{\mathrm{NA}} (L)$ represented by $(\mathcal{X} , \mathcal{L})$; see \cite[Section 3.1]{MR3985614} for more details. Note also that the choice of a divisor $D \subset X$ naturally defines a log test configuration $((\mathcal{X}, \mathcal{L}); \mathcal{D})$ as in Definition \ref{dflgtc}. By (\ref{eqsbgfeg}) and Theorem \ref{thmbhj} we find
	\begin{align}
		0 &\le \lim_{t \to + \infty} \frac{\nu_{\beta} (\phi_t)}{t}  \notag \\
		&= \frac{1}{V} \left( K_{( \bar{\mathcal{X}} , (1 - \beta ) \bar{\mathcal{D}}) / \mathbb{P}^1} \cdot \bar{\mathcal{L}}^n + \frac{\underline{S}_{\beta}}{n+1} \bar{\mathcal{L}}^{n+1} \right) - \frac{1}{V} \left( \mathcal{X}_0 - \mathcal{X}_{0, \mathrm{red}} \right) \cdot \bar{\mathcal{L}}^n \notag \\
		&\le \frac{1}{V} \left( K_{( \bar{\mathcal{X}} , (1 - \beta ) \bar{\mathcal{D}}) / \mathbb{P}^1} \cdot \bar{\mathcal{L}}^n + \frac{\underline{S}_{\beta}}{n+1} \bar{\mathcal{L}}^{n+1} \right) \notag \\
		&=DF (\mathcal{X}, \mathcal{D} , \mathcal{L} , \beta ), \notag
	\end{align}
	where the inequality in the third line follows from $\mathcal{X}_0 - \mathcal{X}_{0, \mathrm{red}}$ being an effective divisor supported on the central fibre and $\bar{\mathcal{L}}$ being $\pi$-semiample on $\bar{\mathcal{X}}$. The last line follows from \cite[Theorem 3.7]{MR3403730}; see also \cite[Definition 3.17]{MR3985614}.
\end{proof}

\subsubsection{Proof of log $K$-polystability}

We improve the above result to prove the log $K$-polystability (\ref{log K polystability}) in \thmref{log K}.
We follow the argument by Berman--Darvas--Lu \cite[Section 4]{MR4094559}. We start by generalising the ingredients of the proof of their result to the conic setting.

\begin{lemma} \emph{(see \cite[Lemma 4.1]{MR4094559})} \label{ldm41bdl}
	Suppose that we fix a smooth K\"ahler metric $\omega$ on $X$. Let $u_0 \in \mathcal{E}^1 \cap D_{\omega}^{-1}(0)$ be a cscK cone potential and $\{ u_t \}_{t \ge 0} \subset \mathcal{E}^1 \cap D_{\omega}^{-1}(0)$ be a finite energy geodesic ray emanating from $u_0$, where $D_{\omega}$ is the functional defined in (\ref{defmaeng}). If there exists $C>0$ such that
	\begin{equation} \label{eqinfjubc}
		\inf_{g \in G} J^A_{\omega_0} (g \cdot u_t) < C
	\end{equation}
	for all $t \ge 0$ for the functional defined in \eqref{defaubjfc}, there exists a real holomorphic Hamiltonian vector field $v \in \mathfrak{isom}(X , \omega_{u_0})$ preserving $D \subset X$ such that $u_t = \exp (t J v) \cdot u_0$, where $J$ is the complex structure on $X$.
\end{lemma} 

\begin{rem} \label{rminfjubc}
	Inspection of the proof below immediately gives that the right hand side of (\ref{eqinfjubc}) only needs to be of order $C + o(t)$, where $o(t)$ stands for the (nonnegative) quantity which goes to zero after dividing by $t$ and taking the limit $t \to + \infty$, i.e.~$\lim_{t \to + \infty}o(t)/t = 0$.
\end{rem}

Before we start the proof, we first recall the definition of the H\"older space $C^{4, \alpha, \beta}$, where $0< \alpha <1$ is a (fixed) H\"older exponent, introduced in \cite[Section 4]{MR4020314} for any $0<\beta\leq 1$, (half angle case $0<\beta<\frac{1}{2}$ in \cite[Section 2]{MR3857693} and \cite[Section 2]{MR3968885}), in such a way that generalises the H\"older space $C^{2 , \alpha , \beta}$ defined by Donaldson \cite{MR2975584}. While we do not give a precise definition of it in this paper, we only point out that the potential functional $u_0$ of a cscK cone metric of angle $2 \pi \beta$ is an element of $C^{4, \alpha, \beta}$, which in particular implies that $u_0$ is a continuous function on $X$ and that there exists a constant $\tilde{C}>0$ such that
\begin{equation} \label{eqchvlfest}
		 \tilde{C}^{-1} |s|^{2 \beta -2} \omega_0^n \le \omega^n_{u_0} \le \tilde{C} |s|^{2 \beta -2} \omega_0^n \quad \text{on } X \setminus D ,
\end{equation}
where $\omega_0$ is reference K\"ahler metric which is smooth on $X$, as in Section \ref{Log K stability}, and $s \in H^0 (X , L_D )$ is the global section that defines $D$ by $D = \{ s=0 \}$ (its norm $|s|$ is computed with respect to some fixed Hermitian metric on $L_D$). It is also well-known that we have $C^{4, \alpha, \beta} \subset \mathcal{E}^1$ when $0 < \beta < 1$. We only need these properties of $C^{4, \alpha, \beta}$ in what follows, and the reader is referred for more details to \cite{MR4020314} and the references above.

\begin{proof}
Suppose that we take a sequence $\{ g_k \}_{k \in \mathbb{N}} \subset G$ such that $J_{\omega} (g_k \cdot u_k) < C$. By \cite[Theorem 6.1]{MR4020314}, $\mathfrak{g} = \mathfrak{aut} ((X,L); D)$ is the complexification of the isometry group $\mathfrak{isom}(X , \omega_{u_0})$ of the cscK cone metric $\omega_{u_0}$, where $u_0 \in C^{4, \alpha , \beta}$. Thus by the global Cartan decomposition (see also \cite[Proposition 6.2]{MR3600039}) we may write $g_k = h_k \exp (-Jv_k)$ for some $h_k \in \mathrm{Isom}_0 (X , \omega_{u_0})$ and $v_k \in \mathfrak{isom}(X , \omega_{u_0})$. We then find
	\begin{equation*} 
		k-C < d_1 (u_0 , \exp(Jv_k) \cdot u_0) < k+C
	\end{equation*}
	exactly as in \cite[Proof of Lemma 4.1]{MR4094559}, which in turn relies on \cite[Proposition 5.5, Lemmas 5.9 and 5.10]{MR3600039}.

	We now claim that the re-scaled vector fields $\{ Jv_k /k \}_{k \in \mathbb{N}}$ contain a convergent subsequence in $C^{\infty}$. We first write $\psi_k \in \mathcal{H} \cap D_{\omega}^{-1}(0)$ for the potential satisfying $\exp(Jv_k)^*\omega = \omega + dd^c \psi_k $, which clearly equals the holomorphy potential of $Jv_k$. We then find $\exp(Jv_k) \cdot u_0 = \psi_k + \exp(Jv_k)^* u_0$, as pointed out in \cite[Lemma 5.8]{MR3600039}. The above inequality and \cite[Theorem 3]{MR3406499} (see also \cite[Proposition 5.4]{MR3600039}) implies that there exists $C_1 >1$ such that
	\begin{align*}
		&C_1^{-1} (k-C) \\
		&< \int_X | \psi_k + \exp(Jv_k)^* u_0 -u_0| \omega_{u_0}^n  \\
		&\quad \quad + \int_X | \psi_k + \exp(Jv_k)^* u_0 -  u_0| (\exp(Jv_k)^* \omega_{u_0})^n \\
		&< C_1 (k + C).
	\end{align*}
	Noting that $\Vert \exp(Jv)^* u_0 -u_0 \Vert_{C^0} < \max_X u_0 - \min_X u_0$ and that $u_0$ is continuous on $X$ as $u_0 \in C^{4 , \alpha , \beta}$, we find that there exists $C_2 >1$ such that for all large enough $k$ we have
	\begin{equation*}
		C_2^{-1} < \frac{1}{k} \int_X | \psi_k | \omega^n < C_2,
	\end{equation*}
	since $\psi_k$ vanishes on $D \subset X$ with order at least 1 as it is a holomorphy potential of a holomorphic vector field which preserves $D$, and
	\begin{equation*}
		 C_3^{-1} |s|^{2 \beta -2} \omega^n \le \omega^n_{u_0} \le C_3 |s|^{2 \beta -2} \omega^n
	\end{equation*}
	for some $C_3 >1$ (and similarly for $(\exp(Jv_k)^* \omega_{u_0})^n$), again by $u_0 \in C^{4 , \alpha , \beta}$ and (\ref{eqchvlfest}). Note that the subset of $\mathcal{H} \cap D_{\omega}^{-1}(0)$ consisting of holomorphy potentials for $\mathfrak{g}$ is a finite dimensional vector space isomorphic to $\mathfrak{g}$. Thus the uniform bound as above means that the sequence $\{ \psi_k /k \}_{k \in \mathbb{N}}$ is contained in a compact subset of $\mathfrak{g}$ and hence there exists $v \in \mathfrak{isom}(X , \omega_{u_0})$, $v \neq 0$, such that $v_k/k$ converges smoothly to $v$ as $k \to \infty$. We then argue as in \cite[Proof of Lemma 4.1]{MR4094559}, which relies on \cite[Proposition 5.1]{MR3687111}, to show that $u_t = \exp(tJv) \cdot u_0 \in C^{4 , \alpha , \beta}$. 
\end{proof}

\begin{proof}[Proof of (\ref{log K polystability}) in \thmref{log K}, log $K$-polystability]
With the above ingredients, the proof is a word-by-word repetition of \cite[Proof of Theorem 1.6]{MR4094559} by Berman--Darvas--Lu, but we provide a sketch proof for the sake of the reader's convenience.

Suppose that we pick a normal test configuration $(\mathcal{X} , \mathcal{L})$ and a smooth subgeodesic ray $(\phi_t)_{t \ge 0}$, emanating from the conic cscK potential, which extends to the total space of $(\mathcal{X} , \mathcal{L})$, just as in Section \ref{scpflkss}. Recalling that (\ref{log K semistability}) in \thmref{log K} implies $DF (\mathcal{X}, \mathcal{D} , \mathcal{L} , \beta ) \ge 0$, it remains to show that $((\mathcal{X}, \mathcal{L}); \mathcal{D})$ is product if  $DF (\mathcal{X}, \mathcal{D} , \mathcal{L} , \beta ) = 0$, in order to prove (\ref{log K polystability}) in \thmref{log K}.

The log properness theorem (Theorem \ref{cone properness theorem}) implies that we have
\begin{equation*}
	0 \le \lim_{t \to + \infty} \frac{\nu_{\beta} (\phi_t)}{t} \le DF (\mathcal{X}, \mathcal{D} , \mathcal{L} , \beta ) = 0,
\end{equation*}
as we saw in Section \ref{scpflkss}. Taking a finite energy geodesic ray $(\psi_t)_{t \ge 0}$ which satisfies (\ref{eqsbgfeg}), we thus find $\nu_{\beta} (\psi_t) = o(t)$ in the notation of Remark \ref{rminfjubc}. Combined with \thmref{cone properness theorem}, we see that there exists a constant $C$ such that
\begin{equation*}
		\inf_{g \in G} J^A_{\omega_0} (g \cdot (\psi_t - D_{\omega}(\psi_t) ) ) < C + o(t)
\end{equation*}
holds for all $t \ge 0$, noting that $J^A_{\omega_0}$ is invariant under an additive constant. We then invoke Lemma \ref{ldm41bdl} (and also Remark \ref{rminfjubc}) and conclude that there exists $v \in \mathfrak{isom}(X , \omega_{u_0})$ such that $\psi_t = \exp (t J v) \cdot \psi_0$.

We observe that \cite[Lemma 4.2]{MR4094559} applies to the situation under consideration since $\psi_t = \exp (t J v) \cdot \psi_0 \in C^{4 , \alpha , \beta} \subset \mathcal{E}^1$ for all $t \ge 0$, which implies that $v$ lifts to the total space of the line bundle $L$. We further apply \cite[Proposition 4.3]{MR4094559}, which was originally proved by Berman \cite[Lemma 3.4]{MR3461370}, to conclude that $v$ is the generator of an (algebraic) $\mathbb{C}^*$-action and that $\mathcal{X}$ is isomorphic to $X \times \mathbb{C}$, as required.
\end{proof}


The necessary direction for the YTD Conjecture \ref{Uniform log YTD necessary} in terms of the uniform stability is obtained by applying \cite[Theorem 1.2 (Theorem 7.4)]{arXiv:1803.09506} and \cite[Corollary 4.5]{MR3985614}. 

\subsubsection{Proof of $G$-uniform log $K$-stability} \label{scpfgulks}

Recalling the definition of the $G$-uniform log $K$-stability from Section \ref{Uniform log K stability}, we prove the $G$-uniform log $K$-stability by following the argument of Hisamoto \cite{arXiv:1610.07998} and C.~Li \cite{arXiv:1907.09399}.

\begin{proof}[Proof of (\ref{uniform log K-stability}) in \thmref{log K}, $G$-uniform log $K$-stability]
This is a consequence of the $d_{1,G}$-coercivity of the log $K$-energy (\ref{d1gcoercive}) proved in \cite{arXiv:1803.09506} (see \thmref{cone properness theorem}). Recall first that the existence of a cscK cone metric implies that the automorphism group $G:=\mathrm{Aut}_0 ((X,L);D)$ is reductive by \cite[Theorem 6.1]{MR4020314} (see also \cite[Theorem 4.1]{MR3968885} for the half angle case $0<\beta<\frac{1}{2}$). We thus write it as a complexification of its maximal compact subgroup $K$, which we identify with the isometry group of the cscK cone metric, and write $Z$ for the identity component of the centre of $K$. We further define $T$ as the complexification of $Z$. Suppose that we write $N_G(K)$ for the normaliser of $K$ in $G$. Inspecting the proof of \cite[Theorem 7.4]{arXiv:1803.09506}, and noting \cite[Proof of Claim 7.9]{MR3600039} by Darvas--Rubinstein and \cite[Proof of Theorem 4.3]{arXiv:1610.07998} by Hisamoto, we find that it implies a stronger inequality
\begin{equation*}
	\nu_\beta(\vphi)\geq A \cdot d_{1, N_G(K)}(\vphi,0)-B
\end{equation*}
for any $\vphi \in \mathcal{E}^1$, where $A >0$ and $B \in \mathbb{R}$ are fixed real numbers. This in turn implies that we have
\begin{equation*}
	\nu_\beta(\vphi)\geq A \cdot d_{1, T}(\vphi,0)-B
\end{equation*}
for any $K$-invariant K\"ahler potential $\vphi \in (\mathcal{E}^1)^K$, where $(\mathcal{E}^1)^K$ is a totally geodesic subspace of $\mathcal{E}^1$ consisting of $K$-invariant potentials, by \cite[Proposition B.1]{arXiv:1907.09399} which is due to J.~Yu.

Pick a $G$-equivariant log test configuration $((\mathcal{X}, \mathcal{D} ); \mathcal{L} )$. Then, as we did in Section \ref{scpflkss}, we can take a smooth subgeodesic ray $(\phi_t)_{t \ge 0} \subset \mathcal{E}^1$ which admits $\phi^{\mathrm{NA}} \in \mathcal{H}^{\mathrm{NA}}(L)$ as non-Archimedean limit, where $\phi^{\mathrm{NA}}$ is represented by $(\mathcal{X}, \mathcal{L})$ (recall that $\mathcal{D}$ is automatically determined by $(\mathcal{X}, \mathcal{L})$ and $D \subset X$). By taking the average over $K$ if necessary, we may assume that $(\phi_t)_{t \ge 0} \subset (\mathcal{E}^1)^K$ and that its non-Archimedean limit is still represented by $(\mathcal{X}, \mathcal{L})$, since $((\mathcal{X}, \mathcal{D} ); \mathcal{L} )$ (and hence $(\mathcal{X}, \mathcal{L})$) is assumed to be $G$-equivariant. Then the argument is very similar to what we did in the proof of (\ref{log K semistability}) in Theorem \ref{log K} (see Section \ref{scpflkss}): we divide (\ref{d1gcoercive}) by $t$ and take the limit $t \to + \infty$, to get
\begin{equation*}
	DF (\mathcal{X}, \mathcal{D} , \mathcal{L} , \beta ) \ge \lim_{t \to + \infty} \frac{\nu_{\beta} (\phi_t)}{t} \ge A \cdot \lim_{t \to + \infty} \frac{d_{1,T} (\phi_t)}{t}.
\end{equation*}

Finally, the result of Hisamoto \cite[Theorem B]{arXiv:1610.07998} (see also \cite[Theorem 3.14]{arXiv:1907.09399} by C.~Li) implies that the limit of the right hand side agrees with $J^{\mathrm{NA}}_T (\mathcal{X} , \mathcal{L})$. We thus conclude, for any $G$-equivariant log test configuration $(\mathcal{X}, \mathcal{D} , \mathcal{L} , \beta )$, that
\begin{equation*}
	DF (\mathcal{X}, \mathcal{D} , \mathcal{L} , \beta ) \ge A' J^{\mathrm{NA}}_T (\mathcal{X} , \mathcal{L})
\end{equation*}
for some $A' >0$ as required, by noting that the $d_1$-distance on $\mathcal{E}^1$ is Lipschitz equivalent to the functional $J^A_{\om_0}$ in (\ref{defaubjfc}) \cite[Proposition 5.5]{MR3600039}.
\end{proof}

\subsubsection{Proof of uniform log $K$-stability}

Given all the results established so far, it is easy to establish the final item of \thmref{log K}.

\begin{proof}[Proof of (\ref{log K stability}) in \thmref{log K}, uniform log $K$-stability]
We simply repeat the proof in section \ref{scpfgulks} when $G = \mathrm{Aut}_0 ((X,L);D)$ is trivial. Alternatively, we can also use the the coercivity of $\nu_\beta$ in the proof of (\ref{log K semistability}) \thmref{log K}, given in section \ref{scpflkss}. The coercivity of $\nu_{\beta}$ implying the uniform log $K$-stability is also written in \cite[Corollary 4.5]{MR3985614}.
\end{proof}


\begin{rem}\label{General Kahler class}
We believe that \thmref{log K} holds for a general K\"ahler class that is not necessarily the first Chern class of an ample line bundle, by adapting the argument of Dervan--Ross \cite{MR3696600} and Sj\"{o}str\"{o}m Dyrefelt \cite{MR4095424,MR3881961}, although we decide not to discuss this problem further in this paper as it will involve subtle technical points.
\end{rem}

\subsection{Openness of stabilities along cscK cone path}

Then combined with \thmref{openess}, (\ref{log K stability}) in \thmref{log K} directly implies that
\begin{prop}\label{openess semistability}
If $\mathrm{Aut}_0 ((X,L);D)=0$ and there exists a cscK cone metric with cone angle $2\pi\beta_0\in (0,2\pi)$, then there is a constant $\delta>0$ such that $((X,L);D)$ is uniformly log $K$-stable with angle $2 \pi \beta$ for all $\beta\in (\beta_0-\delta,\beta_0+\delta)$. Especially, when $\beta=1$, then interval becomes $(1-\delta,1]$.
\end{prop}

The following conclusion follows directly from the linearity of the log $K$-energy \thmref{Log K energy linear}.
\begin{prop}\label{proper interpolation}
	Assume that the log $K$-energy is proper at cone angle $\beta_0>0$ and bounded below at cone angle $\beta_1>\beta_0$. Then $((X,\Omega);D)$ is the log $K$-energy is proper for any $\beta\in [\beta_0,\beta_1)$.  
\end{prop}
Making use of \thmref{cone properness theorem} and \ref{uniform log K-stability} in \thmref{log K}, we also have
\begin{prop}\label{existence interpolation}
Assume that $((X,\Omega);D)$ admits a cscK cone metric with cone angle $\beta_0>0$ and the log $K$-energy is bounded below at cone angle $\beta_1>\beta_0$. Then for any $\beta\in [\beta_0,\beta_1)$
\begin{itemize}
\item $((X,\Omega);D)$ is $d_{1,G}$-proper ;
\item $((X,\Omega);D)$ admits a cscK cone metric;
\item $((X,\Omega);D)$ is $G$-uniformly log $K$-stable with angle $2\pi\beta$.
\end{itemize}
\end{prop}


\subsection{Comparison between maximal cone angles}
The results above directly lead to an estimate of the maximal cone angle $\mathfrak \beta_{cscKc}$ of cscK cone metrics, see Definition \ref{maximal existence cone}.
\begin{defn}
For the smooth polarised pair $((X,L);D)$, we define the following invariant,
\begin{align*}
\mathfrak \b_{lKs}((X,L);D):=\sup_{\beta}\{((X,L);D)\text{ is log K-stable with angle }2\pi\beta\}.
\end{align*}
\end{defn}
Clearly, $((X,L);D)$ is log $K$-stable with cone angle $2\pi\beta
<2\pi\beta_{lKs}(X,L);D)$ and log $K$-unstable with cone angle $2\pi\beta
>2\pi\beta_{lKs}(X,L);D)$

\begin{defn}\label{maximal uniformly log K-stable}
We define the supremum of the cone angle such that 
\begin{align*}
\mathfrak \b_{ulKs}:=\sup_{\beta}\{((X,L);D)\text{ is G-uniformly log K-stable with angle }2\pi\beta\}.
\end{align*}
\end{defn}
\begin{defn}\label{maximal cane angle interval}
We could also define a ``stronger'' version of the maximal cone angle $\mathfrak \b_{lKs}$ such the log $K$-stability is satisfied on the whole interval $(0,\beta)$, i.e. 
\begin{align*}
\overline{\mathfrak \b_{lKs}}=\sup_{\gamma}\{((X,L);D)\text{ is log K-stable with angle }2\pi\beta \text{ for all }\beta\in(0,\gamma]\} .
\end{align*}
and define similarly for the maximal cone angle $\overline{\mathfrak \beta_{cscK}}$ of existence of cscK cone metrics and the maximal cone angle $\overline{\mathfrak \beta_{ulKs}}$ of the $G$-uniform log K-stability.
\end{defn}
\thmref{log K} implies that
\begin{thm}
$
\mathfrak \beta_{cscKc}\leq \mathfrak \b_{ulKs}\leq \mathfrak\beta_{lKs}, \quad \overline{\mathfrak \beta_{cscKc}}\leq \overline{\mathfrak \b_{ulKs}}\leq \overline{\mathfrak\beta_{lKs}}.
$
\end{thm}
The positive answer to the following question, which is equivalent to proving Conjectures \ref{log YTD necessary} and \ref{Uniform log YTD necessary}, is expected.
\begin{ques}\label{cone angle comparison}
$
\mathfrak \beta_{cscKc}= \mathfrak \b_{ulKs}?\quad  \overline{\mathfrak \beta_{cscKc}}= \overline{\mathfrak \b_{ulKs}}?
$
\end{ques}
We could also compare these two types of maximal cone angles.
\begin{ques}\label{compare 2 types of angles}
$
\beta_{a}=\overline{\beta_a}, \quad a \in \{cscKc, lKs, ulKs\}?
$
\end{ques}

Note that the answer to Question \ref{compare 2 types of angles} is indeed affirmative for $a \in \{lKs, ulKs\}$ if we \textit{assume} that $((X,L);D)$ is log $K$-semistable with angle $0$. Indeed, we can argue exactly as in Proposition \ref{existence interpolation} by noting
\begin{equation*}
	DF (\mathcal{X}, \mathcal{D} , \mathcal{L} , \beta ) = (1-s)DF (\mathcal{X}, \mathcal{D} , \mathcal{L} , \beta_1 ) + s DF (\mathcal{X}, \mathcal{D} , \mathcal{L} , \beta_2 ),
\end{equation*}
for $\beta=(1-s)\beta_1+s\beta_2 \in (\beta_1 , \beta_2)$, analogously to Theorem \ref{Log K energy linear}.


\subsection{Average scalar curvature and log $K$-instability} \label{scavsclki}

In this section we study how the value of the average scalar curvature on $D$ affects the log $K$-stability. Let $(X,L)$ be an $n$-dimensional polarised manifold as before. \textit{Assume} that there is a smooth hypersurface $D$ satisfying $D \in |L|$. We write $\tilde{L}$ for the ample line bundle on $D$ defined by $\tilde{L}:= L|_D$. 
Recall that the average scalar curvature $\underline{S}_1$ of a K\"{a}hler metric on $X$ in the K\"ahler class $C_1 (L )$ can be computed as
\begin{equation} \label{defbarsav}
\underline{S}_1 := \frac{n \int_X C_1 (-K_{X})  C_1 (L)^{n-1}}{\int_X C_1 (L)^{n}},
\end{equation}
and the average scalar curvature $\underline{S}^D$ of a one on $D$ in the class $C_1 (\tilde{L})$ can be computed as
\begin{equation} \label{defbarsavd}
\underline{S}^D := \frac{(n-1) \int_D C_1 (-K_{D})  C_1 (\tilde{L})^{n-2}}{\int_D C_1 (\tilde{L})^{n-1}},
\end{equation}
where $-K_{D}$ is the anticanonical line bundle over $D$.

The main result of this section is the following sufficient condition for the log $K$-instability in terms of $\underline{S}^D$.

\begin{thm}
\label{destabilizing test configuration}
Suppose that $D \in |L|$ is smooth. Then $((X,L);D)$ is log $K$-unstable with angle $2 \pi \beta$ if $\beta$ satisfies 
\begin{equation*}
\beta < \frac{\underline{S}^D}{n(n-1)}.
\end{equation*}
\end{thm}

\begin{rem}
In the theorem above and in the rest of this section, the cone angle $\beta$ may take any real value and may not be constrained in the interval $(0,1)$, unlike what is discussed previously.	
\end{rem}

\begin{rem}
	The argument in this section works even when $X$ and $D$ are singular, as long as $X$ is $\mathbb{Q}$-Gorenstein (i.e.~$K_X$ is $\mathbb{Q}$-Cartier) and there exists $D \in |L|$ such that $D$ is $\mathbb{Q}$-Gorenstein and $\tilde{L}$ is $\mathbb{Q}$-Cartier on $D$. We present here only the smooth case, however, since assuming $\tilde{L}$ to be $\mathbb{Q}$-Cartier on $D$ seems slightly artificial. See Section \ref{scnormvar} for more discussions on singular varieties.
\end{rem}

The above theorem can be regarded as partially complementing Theorem \ref{thmalphainv}, in the sense that we have therein $\underline{S}^D \ll 0 < \beta $ when the degree of $D$ is very large. We also recall a well-known result of Sun \cite{MR3004583}.

\begin{thm}[Sun \cite{MR3004583}]
\label{strictly K-semistable}
If $C_1 (\tilde{L})$ admits a scalar-flat {K\"{a}hler} metric, then $((X,L);D)$ is strictly log $K$-semistable with angle $\beta = 0$.\
\end{thm}

We note that Theorem \ref{strictly K-semistable} was stated in \cite{MR3004583} when $X$ is a Fano manifold and $D$ is a smooth anticanonical divisor, but the proof carries over word-by-word to the general polarised case, which is also pointed out in \cite[page 5527]{MR4267616}.

Combined with Theorem \ref{thmalphainv} proved later, the above theorems indicate that $((X,L);D)$ is log $K$-stable if $\beta$ is very large compared to $\underline{S}^D$, strictly log $K$-semistable if $\beta = 0$ (and $D$ has a cscK metric with $\underline{S}^D = 0$), and log $K$-unstable if $\beta$ is smaller than $\underline{S}^D / n(n-1)$.


\begin{rem}
	As pointed out in \cite{MR3004583}, Theorem \ref{strictly K-semistable} can be regarded as an algebro-geometric counterpart of the differential-geometric result of Tian--Yau \cite{MR1040196} which shows the existence of a complete Ricci-flat {K\"{a}hler} metric on $X \setminus D$, if we consider the case $L = -K_X$ for a Fano manifold $X$. Note that in the case of Tian--Yau \cite{MR1040196}, Yau showed that $D$ admits a Ricci-flat {K\"{a}hler} metric \cite{MR480350}. 

From this observation, the following problem arises: if $D$ has a constant scalar curvature {K\"{a}hler} metric with $\underline{S}^D = 0$, does $X \setminus D$ admits a complete scalar-flat {K\"{a}hler} metric? This problem is another version of \cite{Aoi} which deals with the case that $0 < 3 \underline{S}^D <n(n-1)$. Note that Auvray \cite{MR3716943} established a version of the converse of this problem for K\"ahler metrics of Poincar\'e type.
\end{rem}

Before we start the proof of Theorem \ref{destabilizing test configuration}, we observe the following proposition.
\begin{proposition}
\label{scalar curvature bound}
For any polarised K\"ahler manifold $(X , L)$, we have
\begin{equation*}
\underline{S}_1 \leq n(n+1).
\end{equation*}
\end{proposition}

\begin{proof}
	This is a consequence of Mori's cone theorem \cite{MR662120} that was also used in the proof of Proposition \ref{ppautmori}. Here we use that $K_X + (n+1) L$ is a nef divisor on $X$ \cite[Theorem 1.5.33 and Example 1.5.35]{MR2095471}. This implies that we have
	\begin{equation*}
		\int_X C_1 (L)^{n-1} C_1(K_X + (n+1) L) \ge 0,
	\end{equation*}
	by noting that there exists $m>0$ such that the Poincar\'e dual of the cohomology class $m C_1 (L)^{n-1}$ contains a smooth complete intersection curve in $X$, by a repeated application of Bertini's theorem.
\end{proof}

It seems natural to ask the following question concerning the equality case of Proposition \ref{scalar curvature bound}, which can be seen as a generalisation of Fujita's result \cite{MR3783213} to general polarised varieties, but we decide not to discuss it any further in this paper.

\begin{problem}
Suppose that $(X,L)$ admits a constant scalar curvature {K\"{a}hler} metric and the equality $\underline{S}_1 = n(n+1)$ holds. Is $(X,L)$ necessarily isomorphic to $(\mathbb{P}^n , \mathcal{O} (1))$?
\end{problem}

We now give a proof of Theorem  \ref{destabilizing test configuration}.

\begin{proof}[Proof of Theorem \ref{destabilizing test configuration}]
Recall that $D \in |L|$ and the average values of scalar curvatures $\underline{S}_1$ and $\underline{S}^D$ are given by (\ref{defbarsav}) and (\ref{defbarsavd}). We then observe
\begin{equation*}
\underline{S}^D = (n-1) \left( \frac{\int_X C_1 (-K_X) C_1(L)^{n-1}}{\int_X C_1(L)^{n}}  - 1 \right) = \frac{n-1}{n} \underline{S}_1 - (n-1),
\end{equation*}
by the adjunction formula $K_D = (K_X + D )|_D$.

Following \cite{MR2274514} and \cite{MR3004583}, we recall the construction of the test configuration by using the deformation to the normal cone of $D$.
By blowing up $X \times \mathbb{C}$ along $D \times \{ 0 \}$, we obtain a family $\pi : \mathcal{X} \to \mathbb{C}$.
The exceptional divisor $P \subset \mathcal{X}$ corresponds with $\mathbb{P} (\nu_D \oplus \mathbb{C})$, where $\nu_D$ is the normal bundle of $D$ in $X$.
The central fibre $\mathcal{X}_0$ is given by gluing $P$ to $X$ along $D = \mathbb{P}(\nu_D)$.
The $\mathbb{C}^*$-action on $\mathcal{X}$ is obtained by considering the trivial action on $X$ and the standard $\mathbb{C}^*$-action on $\mathbb{C}$.
Let $\mathcal{D}$ be the proper transform of $D \times \{ 0 \}$.
$\mathcal{D}$ is $\mathbb{C}^*$-equivariant and its intersection with the central fibre is the zero section $\mathbb{P}(\mathbb{C})$.
We define an ample line bundle on $\mathcal{X}$ by $\mathcal{L}_c := L -cP$ for a rational number $c \in (0,1)$ (see \cite{MR2274514}).
Thus, we obtain  $\mathbb{C}^*$-equivariant family $(( \mathcal{X},\mathcal{D});\mathcal{L}_c)$ parametrised by $c \in (0,1)$.
Let $t$ be the standard holomorphic coordinate on $\mathbb{C}$.
Following \cite{MR2274514} (see also \cite{MR3004583}), we have the following decomposition
\begin{eqnarray}
\label{decomposition}
H^0 (\mathcal{X}_0 , \mathcal{L}_c^{\otimes k} )
&=& H^0 (X,L^{\otimes (1-c)k}) \oplus \bigoplus_{i=0}^{ck-1}t^{ck -i} H^0 (D, \tilde{L}^{\otimes (k-i)}).\\\nonumber
\end{eqnarray}
Here, we assume that $ck$ is a sufficiently large integer by taking $k \in \mathbb{N}$ to be sufficiently large and divisible. Note that this is the weight decomposition with respect to the $\mathbb{C}^*$-action above.

Recalling that we have the following short exact sequence
\begin{equation*}
0 \to H^0 (X, L^{\otimes (i-1)}) \to H^0 (X, L^{\otimes i}) \to H^0 (D, \tilde{L}^{\otimes i}) \to 0,
\end{equation*}
when $i \in \mathbb{N}$ is large enough, we have
\begin{align*}
{\rm dim} H^0 (\mathcal{X}_0 , \mathcal{L}_c^{\otimes k} ) &=  {\rm dim} H^0 (\mathcal{X}_0 , \mathcal{L}_c^{\otimes (1-c)k} )  + \sum_{i=0}^{ck-1} {\rm dim} H^0 (D, \tilde{L}^{\otimes (k-i)})\\
&= {\rm dim} H^0 (X ,L^{\otimes ck})
\end{align*}
for all sufficiently large and divisible $k \in \mathbb{N}$. This equality implies the flatness of the family $(( \mathcal{X},\mathcal{D});\mathcal{L}_c)$, and hence $(( \mathcal{X},\mathcal{D});\mathcal{L}_c)$ is a test configuration for $((X,L);D)$.

Let $(( \mathcal{X},\mathcal{D});\mathcal{L}_c)$ be the test configuration for $((X,L);D)$ defined as above by the deformation to the normal cone of $D$.

By the Riemann--Roch theorem, we have
\begin{equation*}
a_0 = \frac{1}{n !} \int_X C_1 (L )^n	
\end{equation*}
and
\begin{eqnarray*}
a_1
&=& \frac{1}{2(n-1)!} \int_X C_1 (-K_X ) \cup C_1 (L )^{n-1} \\
&=& \frac{1}{2(n-1)!} \left( \frac{\underline{S}^D}{n-1} + 1  \right)  \int_X C_1 (L )^n \\
&=& \frac{n a_0}{2}  \left( \frac{\underline{S}^D}{n-1} + 1  \right),
\end{eqnarray*}
in the notation of Section \ref{Test configurations and log K stability}.

Since the weight of the $\mathbb{C}^*$-action is $-1$ on $t$, the decomposition above (\ref{decomposition}) implies that we can compute the total weight as follows:
\begin{align*}
w_k
&= - \sum_{i = 0}^{ck - 1} (ck - i) {\rm dim } \hspace{2pt} H^0 (D, \tilde{L}^{\otimes (k-i)}) \\
&= - \sum_{i = 0}^{ck - 1} (ck - i) \left( (k-i)^{n-1} n a_0 + \underline{S}^D \frac{(k-i)^{n-2}}{2(n-1)!} \int_D C_1 (\tilde{L})^{n-1} + O(k^{n-3})    \right)\\
&= - \sum_{i = 0}^{ck - 1} (ck - i) \left( (k-i)^{n-1} n a_0 + n a_0 \underline{S}^D \frac{(k-i)^{n-2}}{2} + O(k^{n-3})    \right)\\
&= -  n a_0 \sum_{i = 0}^{ck - 1} (ck - i) \left( (k-i)^{n-1} + \frac{\underline{S}^D}{2} (k-i)^{n-2} + O(k^{n-3})  \right).\\
\end{align*}

These terms can be expanded in $k$ as
\begin{equation*}
\sum_{i = 0}^{ck - 1} (ck - i) (k-i)^{n-1} =  k^{n+1} \int_0^c (c-x) (1-x)^{n-1} dx + \frac{c}{2} k^n + O(k^{n-1}),
\end{equation*}
and
\begin{equation*}
 \int_0^c (c-x) (1-x)^{n-1} dx = - \frac{1}{n} \left( \frac{ 1 - (1-c)^{n+1} }{n+1} -c \right).	
\end{equation*}

Thus, we can compute $b_0$ and $b_1$ in the notation of Section \ref{Test configurations and log K stability} as
\begin{align*}
	b_0 &= \left( \frac{ 1 - (1-c)^{n+1} }{n+1} -c \right) a_0 \\
	b_1 &= \frac{n a_0}{2} \left( - c + \frac{ \underline{S}^D}{n-1}  \left( \frac{ 1 - (1-c)^{n} }{n} -c \right) \right).	
\end{align*}

By following \cite{MR3004583}, we have
\begin{equation*}
	H^0 (\mathcal{D}_0, \mathcal{L}^{\otimes k}_c ) = H^0 (\mathcal{D}, \mathcal{L}^{\otimes k}_c ) / t H^0 (\mathcal{D}, \mathcal{L}^{\otimes k}_c ) = t^{ck} H^0 (D , \tilde{L}^{\otimes k}).
\end{equation*}
This implies
\begin{equation*}
	\tilde{a}_0 = \frac{1}{(n-1)!} \int_D C_1 (\tilde{L})^{n-1} = n a_0,
\end{equation*}
and
\begin{equation*}
	\tilde{b}_0 = -c \frac{1}{(n-1)!} \int_D C_1 (\tilde{L})^{n-1} = -c n a_0.
\end{equation*}

Substituting the above in the definition of the log Futaki invariant, we get
\begin{align*}
&DF (\mathcal{X}, \mathcal{D}, \mathcal{L}_c, \beta) = \frac{2(a_1 b_0 -a_0 b_1)}{a_0} + (1 - \beta) \left( \tilde{b}_0 - \frac{\tilde{a}_0 }{a_0} b_0 \right)  \\
&= n a_0 \left(  \frac{ 1 - (1-c)^{n+1} }{n+1} +  \frac{ \underline{S}^D}{n-1} \left( \left( \frac{ 1 - (1-c)^{n+1} }{n+1} \right) - \left( \frac{ 1 - (1-c)^{n} }{n} \right) \right) \right)\\
& \quad + n a_0 \left( (1 - \beta) \left( -c  +    \left( c - \frac{ 1 - (1-c)^{n+1} }{n+1} \right)      \right) \right) \\
&= n a_0 \left( \frac{ 1 - (1-c)^{n+1} }{n+1} \right) \left( \beta  +  \frac{ \underline{S}^D}{n-1} \left( 1 -  \left( \frac{ n+1}{n}  \right) \left( \frac{ 1 - (1-c)^{n}}{ 1 - (1-c)^{n+1} } \right) \right) \right).\\
\end{align*}

Noting that for all $c \in (0,1)$ we have
\begin{equation*}
	\frac{ \underline{S}^D}{n-1} \left( 1 -  \left( \frac{ n+1}{n}  \right) \left( \frac{ 1 - (1-c)^{n}}{ 1 - (1-c)^{n+1} } \right) \right) \in \left( -\frac{ \underline{S}^D}{n(n-1)} , 0 \right),
\end{equation*}
there exists $c \in (0,1)$ such that $DF (\mathcal{X}, \mathcal{D}, \mathcal{L}_c, \beta) < 0$ if $\beta$ is less than $\underline{S}^D/n(n-1)$.
\end{proof}


\section{Results on existence of cscK cone metrics} \label{scrsoexccm}
In this section, we will derive some existence results of cscK cone metrics, the key tool is the criteria of the properness of log $K$-energy, \cite[Proposition 4.22]{arXiv:1803.09506}).
\subsection{CscK cone metrics along ample divisors of large degree}
We will let the divisor stay in some multiple of the linear system of an ample line bundle.
\begin{defn}\label{canonical polarised pair}We say a polarised pair $((X,L);D)$ is a \textbf{standard polarised pair} and denoted by $((X,L);D,m)$, if
\begin{itemize}
\item $X$ is a compact K\"ahler manifold;
\item $L$ is an ample line bundle and denote by $\Om$ the K\"ahler class associated with $L$;
\item $D$ is a smooth effective divisor in the linear system $|mL|$ with $m>0$.
\end{itemize}
\end{defn}

The main result that we prove in this section is the following.

\begin{thm} \label{thmalphainv}
	For any $0 < \beta < 1$, there exists $m_0 \in \mathbb{N}$ which depends only on $X$, $L$, and $\beta$, such that the standard polarised pair $((X,L);D,m)$ admits a constant scalar curvature K\"ahler metric with cone singularities of cone angle $2 \pi \beta$ along $D$, as long as $D$ is a generic member of the linear system $|mL|$, if $m \ge m_0 $.
\end{thm}

Thus, any smooth projective variety admits a cscK cone metric, as long as $D$ is a divisor which is ample, generic, and of sufficiently large degree. This can be regarded as a conic analogue of the results for the twisted cscK metrics \cite[Theorem 1]{MR3980270}, whose algebro-geometric counterpart was proved by Dervan--Ross \cite[Theorems 1.2 and 3.12]{MR3985107}.

In fact, on the algebro-geometric side, it is not surprising that there should be a parallel between twisted cscK metrics and the cscK cone metrics, because Dervan's definition of the twisted $K$-stability \cite[Definition 2.3]{MR3564626} can be regarded as the log $K$-stability for a generic member of the linear system; see \cite[\S 3.4]{MR3564626} for more details. On the other hand, however, the parallel is not so obvious on the analytic side, since a twisted cscK metric is globally smooth on $X$ whereas a cscK cone metric has a subtle boundary condition near the divisor $D$. The above theorem establishes the expected parallel on the analytic side.

\subsection{Log alpha invariant and properness}

We first review the key ingredient in the proof of \thmref{thmalphainv}.

\begin{defn}\label{log alpha invariant}\emph{(\cite[Definition 4.9]{arXiv:1803.09506})}
	The log alpha invariant in $\mathcal H_\beta$ is defined as
	\begin{align*}
	\alpha_\beta:=\sup\{\alpha>0\vert\exists C\text{ such that}\sup_{\vphi\in \mathcal H_\beta(\om_\theta)}\int_M e^{-\alpha (\vphi-\sup_X\vphi)}\om_\theta^n\leq C\}.
	\end{align*}
	Here $\mathcal H_\beta$ is the space of all $\om_\theta$ pluri-subharmonic functions.
The reference metric $\omega_\theta$ is a solution to the equation $$\om_\theta^n=e^{h_0}|s|_h^{2\b-2}\om^n_0,$$ in which, $\omega_0$ is a smooth K\"ahler metric, $h_0$ is a smooth function (see \eqref{h0}), $s$ is the defining section of $D$ and $h$ is a Hermitian metric on $L_D$.
\end{defn}
The definition of the log alpha invariant does not depend on the choice of $\om_0, h_K, h$, and also $\om_\theta$. It is also important to note that there is an equivalent definition (Definition \ref{aglog alpha invariant}) in terms of algebraic geometry, and that it generalises to the case when $X$ has some mild singularities.

The key ingredient in the proof is the following proposition proved in \cite[Proposition 4.22]{arXiv:1803.09506}, which involves the log $\alpha$-invariant.

\begin{prop}\label{properness criteria}\emph{(\cite[Proposition 4.22]{arXiv:1803.09506})} 
Let $C_1(L) =[\om_0]$ be a K\"ahler class. Assume that there is a constant $\eta$ satisfies that
\begin{align}\label{cohomology conditions Kahler}
\left\{
\begin{array}{lcl}
	&(i)&0\leq \eta<\frac{n+1}{n}\alpha_\beta,\\
	&(ii)& C_1(X,D)<\eta C_1(L),\\
	&(iii)& (\ul S_\beta-\eta)C_1(L) <(n-1)C_1(X,D).
\end{array}
\right.
\end{align}
Then the log $K$-energy is $J$-coercive in $\Om$, i.e. there exists a constant $C$ such that
\begin{align}\label{J properness}
\nu_\beta(\vphi)
&\geq \left( \frac{n+1}{n}\alpha_\beta-\eta \right) J_{\om_0}^A(\vphi)-C, \quad \forall \vphi\in \mathcal H_\beta(\om_0).
\end{align}
Here $J^A_{\om_0}$ is Aubin's $J$-functional defined in \eqref{defaubjfc}.
\end{prop}

\subsection{Proof of \thmref{thmalphainv}}
We prove that the conditions (\ref{cohomology conditions Kahler}) are satisfied when $D \in |mL|$ for large enough $m$.

\begin{lem}\label{lmlgmprp preliminary}
The first Chern class for the pair $(X,D)$ is 
\begin{align}\label{eqevcdii}
C_1(X,D)=C_1(X)-m(1-\b) C_1(L) ,
\end{align} and the average scalar curvature for a K\"ahler cone metric becomes
\begin{align}\label{averaged cone scalar curvature}
\ul{S}_\beta
 =\underline S_1-mn(1-\b).
\end{align}
Here $\underline S_1=n \frac{\int_X C_1 (X) \Omega^{n-1}}{\int_X \Omega^n}$ is the average scalar curvature for any smooth K\"ahler metric in $\Omega$, which is a topological constant.
\end{lem}
\begin{proof}
The definition of $L_D$ implies $C_1 (L_D) = mC_1(L)$, which in turn implies that we have
	\begin{align*}
		C_1(X,D)&=C_1(X)-(1-\beta)C_1(L_D) 
	\end{align*}
	which implies \eqref{eqevcdii} and \eqref{averaged cone scalar curvature}, by using
	\begin{equation*}
		\ul S_\beta=n\frac{ \int_X C_1(X,D) C_1(L)^{n-1}}{\int_X C_1(L)^n}  
	\end{equation*}
	and the formula of the average scalar curvature $\ul S_1$.

\end{proof}

\begin{lemma} \label{lmlgmprp}
	There exists $m_2(X, L, \beta) \in \mathbb{N}$ depending only on $X$, the polarisation $L$, and $0 < \beta < 1$, such that for all $m \ge m_2(X, L, \beta)$ the conditions (\ref{cohomology conditions Kahler}) are satisfied with $\eta = 0$, and $C_1(X,D) = C_1 (X) - (1 - \beta ) C_1 (L_D)$ with $L_D = L^{\otimes m}$.
\end{lemma}


\begin{proof}
By \lemref{lmlgmprp preliminary}, $\ul S_\beta\cdot C_1(L)- (n-1) C_1 (X,D)$ equals
	\begin{align}
		&= \left[ \ul S_1-mn (1 - \beta)  \right] \Omega - (n-1) C_1(X). \label{eqevcdiii}
	\end{align}	
	Recalling that $C_1(L)$ is a K\"ahler class, the equation (\ref{eqevcdii}) implies that there exists $m_2(X, L, \beta) \in \mathbb{N}$ such that the second condition of (\ref{cohomology conditions Kahler}) is satisfied for all $m \ge m_2 (X, L, \beta)$ and for $\eta = 0$. 
	
	Likewise, (\ref{eqevcdiii}) implies that the third condition of (\ref{cohomology conditions Kahler}) is satisfied for all $m \ge m_2 (X, L, \beta)$ and for $\eta = 0$.
	
	It thus suffices to check that the first condition of (\ref{cohomology conditions Kahler}) holds for any $0 < \beta < 1$ and any $m \ge m_2 (X, L, \beta )$. By Berman's formula \cite[Proposition 6.2]{MR3107540}, we have
	\begin{align}
		\alpha_\beta &:=\alpha(L,(1-\beta)D) \notag\\
		&=m\alpha(L_D,(1-\beta)D)\notag\\
		&\geq m \min\{ \beta, \alpha(L_D), \alpha(L_D\vert_D) \}\notag\\
		&= \min\{ m\beta, \alpha(L), m\alpha(L_D\vert_D) \}\label{log alpha invariant lower bound},
	\end{align}
	where $\alpha(L_D)$ (resp.~$\alpha(L_D\vert_D)$) is the $\alpha$-invariant of the polarisation $L_D = L^{\otimes m}$ on $X$ (resp.~$L_D \vert_D$ on $D$). It is a foundational result of Tian \cite[Proposition 2.1]{MR894378} (see also H\"ormander \cite[Theorem 4.4.5]{MR1045639}) that we have $\alpha(L) >0$ and $\alpha(L_D\vert_D)>0$, and hence we get $\alpha_{\beta} >0$ as required.
\end{proof}

Given what we have established so far, it is now straightforward to derive the main result of this section.

\begin{proof}[Proof of \thmref{thmalphainv}]
	We recall from Proposition \ref{ppautmori} that there exists $m_1 \in \mathbb{N}$, depending only on $X$ and $L$, such that $\mathrm{Aut}_0 ((X,L); D)=0$ for a generic member $D \in |mL|$ if $m \ge m_1$. We take $m_2 \in \mathbb{N}$, which depends on $X$, $L$, and $\beta$, to be as in Lemma \ref{lmlgmprp}. We set $m_0 := \max \{ m_1 , m_2 \}$, and get \thmref{thmalphainv} as an immediate consequence of \thmref{cone properness conjecture} and Proposition \ref{properness criteria}, (see \thmref{properness criteria thm}).
\end{proof}

\subsection{Uniformly log $K$-stable manifolds}\label{Log stable manifolds}
We collect several other applications of Proposition \ref{properness criteria} to find cscK cone metrics and log stable manifolds.

Proposition \ref{properness criteria} implies properness of log $K$-energy in the K\"ahler class $\Omega$. Then, we have a cscK cone metric in $\Omega$ from the existence result \thmref{cone properness conjecture}. These arguments have been seen in the proof of \thmref{thmalphainv}. We collect this step here, before we write down its further differential-geometric application.


\begin{thm}\label{properness criteria thm}
Let $C_1(L)=[\om_0]$ be a K\"ahler class. Assume that the Condition \eqref{cohomology conditions Kahler} holds.

Then the automorphism group $\mathrm{Aut}_0((X,L);D$) is trivial. 
Moreover, there exists a unique cscK cone metric in $\Om$ and $((X,L);D)$ is also uniformly log $K$-stable.
\end{thm}
\begin{proof}
It suffices to show that if the log $K$-energy is $J$-proper in the sense of \eqref{J properness}, i.e.~$
\nu_\beta(\vphi)
\geq \left( \frac{n+1}{n}\alpha_\beta-\eta \right) J^A_{\om_0}(\vphi)-C$, then the automorphism group is trivial.
Assume that the automorphism group is non-trivial, then there exists a non-zero holomorphic vector field $X$ and it generates a one-parameter group of holomorphic transformation $\sigma_t$.

Given a K\"ahler cone metric $\om$, the automorphism $\sigma_t$ induces a cone geodesic $\sigma_t^\ast\om$ (see \cite[Section 2]{MR4020314} for the definition for cone geodesic). But this cone geodesic generated from the K\"ahler cone metric $\om$ may not have enough regularity, or appropriate asymptotic rate along the divisor to proceed the convexity argument below. The key observation here is that, instead of a general K\"ahler cone metric $\om$, we need to alternately make use of the reference K\"ahler cone metric $\om_\theta$, which solves \eqref{Rictheta}, to generate the cone geodesic we need. This reference metric has nice growth along $D$, due to the sharp asymptotic analysis for complex Monge--Amp\`ere equation, established in \cite{MR3911741}. Following the proof in \cite[Section 6]{MR4020314}, we see that the value of the log $K$-energy remains the same along the cone geodesic $\sigma_t^\ast\om_\theta$. Meanwhile, as shown in \cite[Lemma 3.12]{MR4020314}, the $J$-functional is also strictly convex along the cone geodesic, if $X$ is non-trivial, that leads to the contradiction of the $J$-coercivity of the log $K$-energy. Hence, we complete the proof.
\end{proof}

In the next Section \ref{scnormvar}, these differential-geometric results will be further extended to the algebro-geometric setting for a pair of normal variety $(X,\tri)$. Moreover, we will also explore the Condition \eqref{cohomology conditions Kahler} in a quantitative way in terms of the cone angle $\beta$. 

\begin{rem}
This is a particular case of Proposition 4.37 in \cite{arXiv:1803.09506}, where the cohomology class $\Om$ is allowed to be big and nef, not necessary K\"ahler.
\end{rem}
\begin{rem}
In order to consider the case when the automorphism group $G=\mathrm{Aut}_0((X,L);D$) is non-trivial, it is sufficient to consider the $G$-invariant version $J_G$ of the $J$-functional \eqref{J functional} i.e. $J_G(\cdot)=\inf_{\sigma\in G} J(\sigma \cdot)$, and its corresponding gradient flow, to obtain the same $J_G$-properness \eqref{J properness} for $G$-invariant log $K$-energy.
\end{rem}
Recall the following definitions:
\begin{align*}
C_1(X,D)=C_1(X)-(1-\b)C_1(L_D),\quad
\underline S_\b= n \frac{\int_X C_1(X,D) C_1(L)^{n-1}}{\int_X C_1(L)^n}.
\end{align*}
A related important quantity is the slope $\mu$ (see \cite[Theorem 1.3]{MR3428958}), which is defined as
\begin{equation} \label{dfslopemu}
	\mu = \mu_{\beta} ((X,L);D) := \frac{\underline{S}_{\beta}}{n}.
\end{equation}

We then obtain the following several existence theorems in terms of the sign of $C_1(X,D)$.

When $C_1(X,D)=0$, the uniqueness theorem \thmref{uniqueness of cscK cone metric} implies that it is equal to the known log Calabi--Yau metric (c.f.~\cite{MR3761174} and references therein). In  summary, we have
\begin{cor}\label{log Calabi Yau}
Assume $C_1(X,D)=0$ and $C_1(L)$ is a K\"ahler class.
There exists a unique log Calabi--Yau metric in $C_1(L)$ and $((X,L);D)$ is also uniformly log $K$-stable.
\end{cor}

\begin{cor} \label{negative c1}
Assume $C_1(X,D)\leq0$. For any K\"ahler class $C_1(L)$, if there exists a constant $\eta$ satisfying $0\leq \eta<\frac{n+1}{n}\alpha_\beta$ such that
\begin{align}\label{eta Omega}
(\underline S_\b-\eta)C_1(L)<(n-1)C_1(X,D).
\end{align}
Then $((X,L);D)$ admits a cscK cone metric and is also uniformly log $K$-stable.
\end{cor}

\begin{cor} \label{positive c1}
Assume $C_1(X,D)\geq0$.  If there exists a constant $\eta$ satisfying $0\leq \eta<\frac{n+1}{n}\alpha_\beta$ such that $$C_1(X,D)<\eta C_1(L) \text{ and }\eqref{eta Omega}.$$
Then $((X,\Omega);D)$ admits a cscK cone metric and is also uniformly  log $K$-stable.
\end{cor}

\begin{rem}In the special situation for $D$ staying in some multiple of $-K_X$ i.e.~the K\"ahler--Einstein cone metric, similar results are shown in \cite{MR3403730}.
\end{rem}

\begin{rem}
The parallel results for twisted cscK metric are given in \cite{MR3564626}.
\end{rem}

\subsection{Entropy threshold and $\mathcal{J}$-threshold}
We further discuss an application of the condition \eqref{cohomology conditions Kahler}.
\begin{defn}\label{Lambda}
We set constants $\Lambda, \lambda \in \mathbb{R}$ such that
$$\lambda=\sup_C \{C\cdot C_1(L) \leq C_1(X)\},\quad  \Lambda=\inf_C \{C_1(X)\leq C\cdot C_1(L) \}.$$
\end{defn}

We observe that $\lambda$ is exactly the nef threshold of $L$.

Recall the log $K$-energy \eqref{log K energy} is defined to be
\begin{align}
\nu_\beta(\vphi)
=E_\beta(\vphi)
+J_{-\theta}(\vphi)+\frac{1}{V}\int_M (\mathfrak h+h_0)\om_0^n,
\end{align}
where $
\mathfrak h:=-(1-\b)\log |s|_h^2.
$
Following the same argument as Page 2809 in \cite{MR3412393}, we could add and subtract the term $\eta J_\omega$ for some constant $\eta$ in the log $K$-energy  
\begin{align*}
\nu_\beta(\vphi)
=E_\beta(\vphi)-\eta J_{\omega_\theta}(\vphi)
+J_{-\theta}(\vphi)+\eta J_{\omega_\theta}(\vphi)+\frac{1}{V}\int_M (\mathfrak h+h_0)\om_0^n.
\end{align*}

Firstly, the existence of the global minimiser of the functional $$J_{-\theta,\eta}(\vphi):=J_{-\theta}(\vphi)+\eta J_{\omega_\theta}(\vphi)$$ is then given in \cite{MR3412393}.
The goal there is to obtain the coercivity of the $K$-energy, by using the lower bound of the twisted functional $J_{-\theta,\eta}$. These results were generalised to the conical setting in \cite{arXiv:1803.09506}. 
We collect related results as follows.
\begin{prop}\emph{(\cite[Theorem 4.20, Proposition 4.22]{arXiv:1803.09506})} \label{properness criteria reduced}
Suppose that the constant $\eta$ satisfies the following conditions,
\begin{align}\label{reduced conditions}
\left\{
\begin{array}{lcl}
	&(i)& C_1(X,D)<\eta C_1(L),\\
	&(ii)& \ul S_\beta C_1(L)- (n-1)C_1(X,D)<\eta C_1(L).
\end{array}
\right.
\end{align}
Then the functional $J_{-\theta,\eta}(\vphi)$ has lower bound for any $\vphi\in \mathcal H_\beta$.
\end{prop}

Secondly, from \eqref{entropy} and \eqref{Rictheta pde}, we have
\begin{align*}
E_\beta(\vphi)=\frac{1}{V}\int_M\log\frac{\om^n_\vphi}{\om_\theta^n}\om_\vphi^{n}.
\end{align*}
It lower bound is obtained directly by the well-known inequality
\begin{align}\label{E and J}
E_\beta(\vphi)
\geq \frac{n+1}{n}\alpha_\beta J_{\om_\theta}(\vphi)-C, \quad \forall \om_\vphi\in [\om_\theta].
\end{align}This inequality is obtained from the Jensen inequality, the definition of the log alpha invariant, see Lemma 5 in \cite{MR3412393}.

Considering the optimal constant $\eta$ brings us to the following definition of threshold of both the entropy $E_{\beta}$ and the $J$-functional.
\begin{defn}We set the entropy threshold
\begin{align}\label{defn e}
e:=\sup_{ \eta}\{ \exists C\text{ such that } E_\beta\geq \eta \cdot J_{\om_\theta}-C, \forall \om_\vphi\in [\om_\theta]\}
\end{align}and the $\mathcal J$-threshold to be 
\begin{align}
\mathcal J=\inf_{ \eta}\{\exists C\text{ such that } J_{-\theta}+\eta\cdot J_{\omega_\theta}\geq -C, \forall \om_\vphi\in [\om_\theta]\}.
\end{align}
Here, we have $$J_{\om_\theta}=I_{\om_\theta}^A-J_{\om_\theta}^A.$$
\end{defn}
\begin{rem}
We could use $\om_0$ instead of $\om_\theta$ is the definition above.
Note that $J_{\om_\theta}$ and $J_{\om_0}$ only differ by  a constant depending on $\|\vphi_\theta\|_\infty$. The relation between these different $J$ functionals is $J_{\om_0}=I_{\om_0}^A-J_{\om_0}^A$.
\end{rem}
\begin{rem}
Note that both $E$ and $J$ only depend on the K\"ahler cone metric $\om_\vphi$, not $\vphi$. Precisely, letting $C$ be any constant, we have
$$E(\vphi+C)=E(\vphi),\quad J_{\chi}(\vphi+C)=J_{\chi}(\vphi).$$
Here $\chi$ is a closed form. In contrast, 
$$D_{\om_\theta}(\vphi+C)=D_{\om_\theta}(\vphi)+C.$$
\end{rem}
Moreover, it is direct to see from \eqref{E and J} that
\begin{lem}
The lower bound of $e$ is
$
e\geq \frac{n+1}{n}\alpha_\beta.
$
\end{lem}
Also,
\begin{lem}\label{e and j}
If $e>\mathcal J$, then the log $K$-energy is $J$-coercive.
\end{lem}

In conclusion, we have a corollary of Proposition \ref{properness criteria reduced}.
\begin{prop}
Assume that 
\begin{align*}
e>\max\{\Lambda,\ul S_\beta-(n-1)\lambda\},
\end{align*} then the log $K$-energy is $J$-coercive, moreover, $C_1(L)$ admits a cscK cone metric and is uniformly log $K$-stable.
\end{prop}
\begin{proof}
Since $C_1(X,D)\leq \Lambda C_1(L)$ and
$$\ul S_\beta C_1(L) -(n-1)C_1(X,D)\leq [\ul S_\beta-(n-1)\lambda] C_1(L),$$
the sufficient conditions of $\eta$ to get \eqref{reduced conditions} is to set  $$ \eta> \Lambda ,\quad \eta>[\ul S_\beta-(n-1)\lambda] .$$
Consequently, the threshold $\mathcal J$, which is defined to be the infimum of such $\eta$, has upper bound
\begin{align*}
\mathcal J\leq \max\{\Lambda,\ul S_\beta-(n-1)\lambda\}.
\end{align*}
Then we have obtained the conclusion from \lemref{e and j} and \thmref{properness criteria thm}.
\end{proof}


The following result is a log adaption of \cite[Proposition 4.11]{MR3956691}, which suggests that \eqref{defn e} has an alternative formula.
\begin{prop}
\begin{align}\label{another e}
e=\sup_{ \eta}\{ \exists C\text{ s.t. }  \| e^{-\vphi}\|_{L^\eta(\om_\theta)}\leq C e^{-D_{\om_\theta}(\vphi)}, \quad \forall \om_\vphi\in [\om_\theta]\}.
\end{align}
Here, $\|\cdot\|_{L^\eta(\om_\theta)}=(V^{-1}\int_M|\cdot |^\eta\om_\theta^n)^{\frac{1}{\eta}}$.
\end{prop}
\begin{proof}
From \eqref{another e}, we have 
$
\int_M e^{-\eta[\vphi-D_{\om_\theta}(\vphi)]-\log\frac{\om_\vphi^n}{\om^n_\theta}}\om_\theta^n\leq C.
$
Then Jensen's inequality gives
$
E_\beta(\vphi)+C\geq \frac{\eta}{V}\int_M [-\vphi+ D_{\om_\theta}(\vphi)]\om_\vphi^n
=\eta  J_{\om_\theta}(\vphi),
$ that is \eqref{defn e}. 

For the other direction from \eqref{defn e} to \eqref{another e}, we need to show 
\begin{align}\label{phi inequality}
V^{-1}\int_M e^{-\eta\vphi}\om_\theta^n\leq Ce^{-\eta D_{\om_\theta}(\vphi)}.
\end{align}
We first find an auxiliary function $u$ as suggested in \cite[Lemma 2.11]{MR3956691}. Let us  set $\oint=V^{-1}\int$ and consider the following functional 
\begin{align}\label{defn A}
A(u):=E_\beta(u)+\eta \oint_M\vphi \om^n_u=\oint_M[\log\frac{\om_u^n}{\om_\theta^n}+\eta\vphi ]\om^n_u.
\end{align}
Its first variation is
\begin{align*}
\delta A(u,\delta u):=\oint_M[\log\frac{\om_u^n}{\om_\theta^n}+\eta\vphi  ]\tri_u(\delta u)\om^n_u
\end{align*} and the critical point equation is solvable
\begin{align}\label{auxiliary phi}
\log\frac{\om_u^n}{\om_\theta^n}+\eta\vphi = C_0\text{ with }V=\int_X e^{C_0-\eta\vphi}\om_\theta^n.
\end{align} Moreover, the critical point $\om_{v}$ is a K\"ahler cone metric and has geometric polyhomogeneity, c.f.~\cite{MR3761174,MR3911741}.
That enables us to see that the critical point is a minimiser, i.e.
\begin{align*}
\delta^2 A(v,\delta u):=\oint_M|\tri_{v}(\delta u)|^2\om^n_{v}\geq0.
\end{align*} 

Then we make use of the equations \eqref{auxiliary phi} and \eqref{defn A} of the auxiliary function $\phi$ to see,
\begin{align*}
L.H.S. :=\oint_M e^{-\eta\vphi}\om_\theta^n
\overset{\eqref{auxiliary phi}}{=}  e^{-C_0}
\overset{\eqref{defn A}}{=}e^{-A(v)} .
\end{align*}

We next set the other functional
\begin{align*}
B(u):=J_{\om_\theta}(u)+\oint_M \vphi\om_u^n
\end{align*}
and apply the assumption $E_\beta(v)\geq \eta \cdot J_{\om_\theta}(v)-C$ in \eqref{defn e}, to get 
\begin{align*}
L.H.S. \leq C e^{-\eta B(v)}.
\end{align*}
Recall that $D_{\om_\theta}(\vphi)= J_{\om_\theta}(\vphi)+\oint_M \vphi\om_\vphi^n$. It remains to check that \begin{align*}
B(v)\geq B(\vphi)=D_{\om_\theta}(\vphi).
\end{align*}

At last, we check this inequality by direct computation
\begin{align*}
\delta B(u,\delta u)=n \oint_M \delta u(\om_\vphi- \om_u)\wedge \om_u^{n-1}.
\end{align*} So the critical point is $\om_u=\om_\vphi$. Moreover, it is a minimiser, since
\begin{align*}
\delta^2 B(\vphi,\delta u)= \oint_M| \p (\delta u) |_\vphi^2 \om_\vphi^{n}\geq 0.
\end{align*} Therefore, we have \begin{align*}
L.H.S. \leq C e^{-\eta B(\vphi)}
\end{align*} and \eqref{phi inequality} is proved.
\end{proof}
\subsection{Small angle solution for cscK cone path}
In this section, we choose $\beta$ sufficiently small to prove existence of cscK cone metric with the small angle $\beta$. This is actually an update of \thmref{thmalphainv} in a quantitative way.


\begin{defn}\label{beta u}Let $((X,L);D,m)$ be a standard polarised pair. 
We define the critical cone angle to be the invariant
\begin{align*}
\beta_u((X,L);D,m)= \min \left\{ 1,\frac{n+1}{n}\frac{\min\{ \alpha(L), m\alpha(L_D\vert_D)\} }{m}+1-\frac{\ul S_1}{mn} \right\}.
\end{align*}
\end{defn}
\begin{thm} \label{small angle existence}
Let $((X,L);D,m)$ be a standard polarised pair.
Suppose one of the following conditions holds,
\begin{enumerate}
\item
fix a sufficiently large $m$ such that 
\begin{align}\label{large m}
\ul S_1<mn+(n-1)\lambda,\quad \Lambda<m;
\end{align} 

\item $m \in \mathbb{N}$ is an integer which satisfies
\begin{align}\label{given m}
\ul S_1\leq mn,\quad \Lambda\leq \frac{\ul S_1}{n}\leq  \lambda+m (1-  \beta_u  ).
\end{align} 
\end{enumerate}

Then $((X,L);D,m)$ admits a cscK cone metric of cone angle $2 \pi \beta$ along $D$ for any $\beta$ in the following range for each of two cases above:
\begin{enumerate}
\item 
in case 1, $\beta$ can take any value in the range
\begin{align}\label{case 1 angle}
0<\beta\leq \min \left\{ 1,1-\frac{\Lambda}{m},1-\frac{\ul S_1-(n-1)\lambda}{mn} \right\};
\end{align}
\item
in case 2, the range is
\begin{align}\label{case 2 angle}
0<\beta\leq \beta_u.
\end{align}
\end{enumerate}
\end{thm}
We comment on the condition \eqref{large m}, before we give the proof.
\begin{rem}\label{constraint}
From the second inequality in \eqref{given m}, we see that $\beta_u$ must be less than or equal to 1. That is the reason why $\beta_u$ is defined in the way as in Definition \ref{beta u}.
We also see that when $X$ is Fano and $L = -K_X$, $\lambda=\Lambda=1$. The assumptions in \eqref{given m} are satisfied automatically.
In general, when the difference $\Lambda-\lambda >0 $ is large, the multiplicity $m$ must be very large and we also need $\beta_u <1$.
\end{rem}

\begin{proof}
It is sufficient to verify the Condition \eqref{cohomology conditions Kahler}. Then, the automorphism group is trivial by Proposition \ref{ppautmori}, and also \thmref{properness criteria thm}, Then the conclusion follows from \thmref{properness criteria thm}.

For case 1, we choose $\eta=0$ in Condition \eqref{cohomology conditions Kahler}, that is we need to show that $\alpha_\beta$ is positive, and $C_1(X,D)$ and $\ul S_\beta\cdot C_1(L)- (n-1) C_1 (X,D)$ are both negative.

Making use of the formulas \eqref{eqevcdii} and \eqref{eqevcdiii}, together with the choice of $m$ in \eqref{large m}, we  see that
	\begin{align*}
		C_1(X,D)=C_1(X)- m (1-\beta)  C_1(L) \leq [\Lambda- m (1-\beta) ] C_1(L)
	\end{align*} and
		\begin{align*}
\ul S_\beta\cdot C_1(L) - (n-1) C_1 (X,D)
		&= \left[ \ul S_1-mn (1 - \beta)  \right] C_1(L) - (n-1) C_1(X)\\
		&\leq \left[ \ul S_1-mn (1 - \beta)  - (n-1)\lambda \right] C_1(L) .
	\end{align*}
	are both negative, as long as $\beta$ is sufficiently small and $m$ is sufficiently large. As shown above in the proof of \thmref{thmalphainv}, $\alpha_\beta$ is always positive.

For case 2, we choose $$\eta=\mu+\eps=\frac{\ul S_\beta}{n}+\eps=\frac{\ul S_1}{n}-m(1-\beta)+\eps$$ in Condition \eqref{cohomology conditions Kahler} for some positive constant $\eps$ determined below, where we recall that $\mu$ was defined in \eqref{dfslopemu}. Now we will show that 
\begin{align}\label{3conditions}
&\alpha_\beta>\frac{n}{n+1}\eta, \quad C_1(X,D)< \eta C_1(L), \notag\\
&\ul S_\beta\cdot C_1(L)- (n-1) C_1(X,D)<\eta C_1(L).
\end{align}

Firstly, we use the lower bound \eqref{log alpha invariant lower bound} of the alpha invariant,
$
		\alpha_\beta \geq \min\{ m\beta, \alpha(L), m\alpha(L_D\vert_D) \}.
$ We set $\tilde\alpha=\min\{ \alpha(L), m\alpha(L_D\vert_D)\}$.
The first inequality in \eqref{3conditions} is equivalent to
	\begin{align*}
\tilde\alpha\geq m\beta, \quad
m\beta\geq \frac{n}{n+1}(\mu+\eps)
	\end{align*} or
\begin{align*}
 \tilde\alpha< m\beta, \quad
\tilde\alpha\geq \frac{n}{n+1}(\mu+\eps).
	\end{align*}
In the first situation, once we choose $$n\eps\leq m\beta,$$ the second inequality automatically holds under the assumption that $\ul S_1\leq mn$. So we only need the cone angle satisfies $$\beta\leq \frac{\tilde\alpha}{m}.$$ In the second situation, after inserting the formula of $\mu$ in \eqref{dfslopemu}, we have
$$\frac{\tilde\alpha}{m}<\beta\leq \frac{n+1}{n}\frac{\tilde\alpha}{m}+1-\frac{\ul S_1}{mn}-\frac{\eps}{m}.$$
Putting these two situations together, we have
$$0<\beta\leq \frac{n+1}{n}\frac{\tilde\alpha}{m}+1-\frac{\ul S_1}{mn}-\frac{\eps}{m}.$$

Secondly, we check the last two inequalities in \eqref{3conditions}.
It is sufficient to show
\begin{align*}
 \Lambda- m (1-\beta)< \eta, \quad  \ul S_1-mn (1 - \beta)  - (n-1)\lambda  <\eta.
\end{align*}
Equivalently, 
\begin{align*}
 \Lambda< \frac{\ul S_1}{n}+\eps, \quad \frac{\ul S_1}{n}<\lambda+m(1-\beta)+\frac{\eps}{n-1}.
\end{align*}
According to the second assumption in \eqref{given m}, we see that the first inequity holds and $\frac{\ul S_1}{n}\leq\lambda+m(1-\beta)$ for any $\beta\leq\beta_u$.
So, the second inequality holds for any $\eps>0$. Thus, we complete the proof.

\end{proof}
The first application of this theorem is an answer to Question \ref{compare 2 types of angles}.
\begin{cor}
Let $((X,L);D,m)$ be a standard polarised pair. Under the assumption of \eqref{given m}, we have
\begin{align*}
\beta_{cscKc}\geq \beta_u,\quad \beta_{cscKc}=\overline{\beta_{cscKc}}.
\end{align*}
\end{cor}
\begin{proof}
The first conclusion is direct and the second one follows from Proposition \ref{existence interpolation}.
\end{proof}

As the second application, this theorem also leads to a result for K\"ahler--Einstein cone metric, which was proved in \cite[Corollary 2.19]{MR3248054}.
\begin{cor}
On Fano manifold, if $L=-K_X$, then there exists a K\"ahler--Einstein cone metric, when
$
0<\beta\leq \beta_u(-K_X),
$ for any $m\geq1$.
\end{cor}
\begin{proof}
In this case, we have $\Lambda=\lambda=1$ and $\ul S_1=n$. So, both conditions \eqref{large m} and  \eqref{given m} are satisfied.
\end{proof}

\section{Uniform log $K$-stability of singular varieties} \label{scnormvar}
In this section, we will present an algebro-geometric criteria of uniform log $K$-stabilities on a pair of normal variety $(X,\tri)$, which are the counterpart of the differential-geometric results including \thmref{properness criteria thm} in Section \ref {Log stable manifolds}.

The singular cscK metric on the singular pair $(X,\tri)$ was introduced in \cite{arXiv:1803.09506}, which extend the study of singular K\"ahler--Einstein metrics motivated from the minimal model programme. The existence of the singular cscK metric was shown in \cite[Section 4]{arXiv:1803.09506}, and we hope enthusiastic readers will find more resources there.

\subsection{Preliminaries}
We collect some materials concerning the singularities of a projective variety. The reader is referred to \cite[Section 2.3]{MR1658959} for more details. Let $(X,\tri)$ be a pair consisting of an irreducible normal complex projective variety $X$ and an effective $\mathbb R$-Weil divisor $\tri$ such that $K_X + \tri$ is $\mathbb{R}$-Cartier; note that $K_X$ is well-defined as a Weil divisor since $X$ is normal.
\begin{defn} \label{dfdcplgrs}
A \textit{log resolution} $\pi:\tilde X\rightarrow X$ of $(X,\tri)$ gives
\begin{align*}
\pi^\ast(K_X+\tri)=K_{\tilde X}-D,
\end{align*}
in which $\tri=-\pi_\ast D$ and $D:=\sum_{i}a_iE_i$ is an $\mathbb R$-Weil divisor and $\cup_iE_i$ is a simple normal crossing divisor in $\tilde{X}$.
\end{defn}

Recall that $E$ is said to be a \textbf{prime divisor over} $X$ if there exists a normal variety $Y$ and a proper birational morphism $\pi : Y \to X$ such that $E$ is a prime (i.e.~reduced irreducible Weil) divisor on $Y$. Given a divisor $E$ over $X$, we define the following important quantity.

\begin{defn}
Given a prime divisor $E$ over $X$, we define 
\begin{equation*}
	a (E,X, \tri ) := \mathrm{ord}_E (K_Y - \pi^*(K_X + \tri ))
\end{equation*}
where $\pi : Y \to X$ is a proper birational morphism and $E \subset Y$, with $Y$ normal. We call $a (E,X, \tri )$ the  \textbf{discrepancy} of $(X,\tri)$ along $E$.
\end{defn}
Note that the divisor $E_i$ which appears in Definition \ref{dfdcplgrs} is a divisor over $X$, and $a_i$ therein is precisely $a (E_i ,X, \tri )$. An important point above is that $a (E,X, \tri )$ depends only on $(X, \tri )$ and the divisor $E$ over $X$ (see e.g.~\cite[Remark 2.23]{MR1658959}).

\begin{defn}Several classes of mild singularities are defined as follows.
\begin{itemize}
\item 
The pair $(X,\tri)$ is log canonical, if $a (E,X, \tri ) \geq -1$ holds for any prime divisor $E$ over $X$.
\item
The pair $(X,\tri)$ is \textit{Kawamata log terminal (klt)}, if $a (E,X, \tri )>-1$ holds for any prime divisor $E$ over $X$. \item 
When $\tri=0$, $X$ is said to be log canonical (resp.~log terminal), if $(X,0)$ is log canonical (resp.~Kawamata log terminal).
\end{itemize}
\end{defn}

In what follows we will need the log alpha invariant for a klt pair, defined as follows.

\begin{defn} \label{aglog alpha invariant}
	Fixing $0 < \beta < 1$, let $\tri $ be an effective $\mathbb R$-Weil divisor in $X$ such that $K_X + (1- \beta) \tri$ is $\mathbb{R}$-Cartier and $L$ be an ample line bundle on $X$. The \textbf{log alpha invariant} of the polarised pair $((X , (1 - \beta ) \tri );L)$ is defined by
	\begin{equation*}
		\alpha ((X , (1 - \beta ) \tri );L) := \inf_{m \in \mathbb{N}} \inf_{D_m \in |mL|} \mathrm{lct} \left( ( X , (1 - \beta ) \tri ) ; \frac{1}{m} D_m \right),
	\end{equation*}
	which we also abbreviate as $\alpha_{\beta}$, where $\mathrm{lct}$ stands for the log canonical threshold defined as
	\begin{align*}
		\mathrm{lct} &\left( ( X , (1 - \beta ) \tri ) ; \frac{1}{m} D_m \right) \\
		&:= \sup \left\{ c \in \mathbb{R} \; \left| \; \left( X,  (1 - \beta ) \tri + \frac{c}{m} D_m \right) \text{ is log canonical.} \right\} \right.
	\end{align*}
	We decree $\mathrm{lct} \left( ( X , (1 - \beta ) \tri ) ; \frac{1}{m} D_m \right) = - \infty$ if there does not exist $c \in \mathbb{R}$ such that $\left( X,  (1 - \beta ) \tri + \frac{c}{m} D_m \right)$ is log canonical.
\end{defn}

In what follows, we shall mostly consider the case when $(X, (1 - \beta) \tri )$ is log canonical. Note that $\alpha ((X , (1 - \beta ) \tri );L) \ge 0$ if $(X, (1 - \beta) \tri )$ is log canonical.

When $X$ and $\tri$ are both smooth, the log alpha invariant defined above agrees with the analytic definition given in Definition \ref{log alpha invariant}, as proved in \cite[Proposition A.4]{MR3107540}; see also \cite[Appendix A]{MR2484031} and \cite[Section 5]{MR3428958}. Thus, without loss of generality, $\alpha_{\beta}$ will always stand for the quantity defined above in what follows. Note also that the average scalar curvature (\ref{eqavlgsc}) makes sense as a ratio of intersection numbers
\begin{equation}\label{intersection averaged scalar}
	\underline{S}_{\beta} = n \frac{(-K_X + (1 - \beta ) \tri)L^{n-1} }{L^n}
\end{equation}
over the normal projective variety $X$.

\subsection{Computation of log Donaldson--Futaki invariant}
Let $X$ be a $\mathbb{Q}$-Gorenstein (i.e.~$K_X$ is $\mathbb{Q}$-Cartier) normal projective variety, $\tri$ an effective integral $\mathbb{Q}$-Cartier divisor, which implies that $K_X + (1 - \beta ) \tri$ is $\mathbb{R}$-Cartier for any $\beta \in (0,1)$. We recall some results on the Donaldson--Futaki invariant for the log test configurations of $((X,L); \tri)$.

The log Donaldson--Futaki invariant is computed by blowing up flag ideals, as in Odaka--Sun \cite{MR3403730}. We apply their formula to obtain several criterion between log stable manifolds and singularity types of normal varieties, which could be seen as algebro-geometric analogues of the differential-geometric results in Section \ref{Log stable manifolds}. Parallel results for twisted cscK metrics were obtained in \cite{MR3564626}.

In this section we use the blow-up formalism of test configurations, and the reader is referred to the paper of Odaka--Sun \cite{MR3403730} for more details. Recall \cite[Definition 3.4]{MR3403730} that a coherent ideal $\mathcal{I}$ of $X \times \mathbb{C}$ is called a flag ideal if it is invariant under the natural $\mathbb{C}^*$-action on $X \times \mathbb{C}$. We define $\mathcal{B} := \mathrm{Bl}_{\mathcal{I}} (X \times \mathbb{P}^1) $ and write $\pi : \mathcal{B} \to X \times \mathbb{P}^1$ for the blowdown map; note that $\mathcal{B}$ agrees with the compactification of $\mathrm{Bl}_{\mathcal{I}} (X \times \mathbb{C})$ over $\mathbb{P}^1$ as in Definition \ref{defcptcx}. 
Writing $\mathrm{pr}_1 : X \times \mathbb{P}^1 \to X$ for the natural projection, we also define a semiample line bundle $\mathcal{L}' := \pi^* \mathrm{pr}^*_1 L$ on $\mathcal{B}$. Fixing a divisor $\tri \subset X$, we also have $\mathcal{B}_{\tri} := \mathrm{Bl}_{\mathcal{I}|_{\tri \times \mathbb{P}^1}} (\tri \times \mathbb{P}^1)$. It is well-known \cite[Proposition 3.5]{MR3403730} that for any log test configuration $((\mathcal{X}, \mathcal{L}); \mathcal{D})$ for a polarised pair $((X,L); \tri)$ as above, there exists a flag ideal $\mathcal{I}$ such that the log Donaldson--Futaki invariant of $((\mathcal{B} , \mathcal{L}' - E); \mathcal{B}_{\tri} )$ agrees with that of $((\mathcal{X}, \mathcal{L}); \mathcal{D})$, where $E$ is the exceptional Cartier divisor of $\pi$. Thus without loss of generality we may only consider the log test configurations of the form described above, and moreover we may assume that $\mathcal{B}$ is Gorenstein in codimension 1 (so that $K_{\mathcal{B}}$ is a well-defined Weil divisor) \cite[Corollary 3.6]{MR3403730}. Furthermore, the log Donaldson--Futaki invariant of $((\mathcal{B} , \mathcal{L}' - E); \mathcal{B}_{\tri} )$ admits the following formula given in terms of intersection numbers.

\begin{thm}\label{log DF}\emph{(\cite[Theorem 3.7]{MR3403730})} 
The log Donaldson--Futaki invariant of the blowup $((\mathcal{B} , \mathcal{L}' - E); \mathcal{B}_{\tri} )$ is given by the following formula
\begin{align}\label{Algebraic log Donaldson-Futaki}
DF (\mathcal B , \mathcal{B}_{\tri} ,\mathcal{L}' -E , \beta )\\
 =(\mathcal{L}'-E)^n\cdot &[\frac{\ul S_\beta}{n+1} (\mathcal{L}'-E)+\pi^* ((K_X+(1-\beta)\tri) \times \mathbb{P}^1)+K_e].\nonumber
\end{align}
Here, $K_e:=(K_{\mathcal B/(X,(1-\beta)\tri)\times\mathbb P^1})_{exc}$ denotes the exceptional part of the divisor $K_{\mathcal B/(X,(1-\beta)\tri)\times\mathbb P^1}$, where
\begin{equation} \label{dfrellgcb}
	K_{\mathcal B/(X,(1-\beta)\tri)\times\mathbb P^1} := K_{\mathcal B}-\pi^\ast( (K_{X}+(1-\beta)\tri )\times \mathbb{P}^1),
\end{equation}
and we also note that $(K_X+(1-\beta)\tri) \times \mathbb{P}^1 = \mathrm{pr}_1^* (K_X+(1-\beta)\tri) $ is an $\mathbb{R}$-Cartier divisor on $X \times \mathbb{P}^1$.
\end{thm}
As mentioned above, the log Donaldson--Futaki invariant $DF (\mathcal{X}, \mathcal{D} , \mathcal{L} , \beta )$ for the test configuration $((\mathcal{X}, \mathcal{L}); \mathcal{D})$ is equal to $DF (\mathcal B , \mathcal{B}_{\tri} ,\mathcal{L}' -E , \beta)$.
\subsection{Uniform log $K$-stability for a singular pair $(X,\tri)$}

We assume as before that $X$ is a $\mathbb{Q}$-Gorenstein normal projective variety, and $\tri$ is an effective integral $\mathbb{Q}$-Cartier divisor.

\subsubsection{Log Calabi--Yau pair $C_1(X,\tri)=0$}

\begin{defn}
We say that $(X, (1 - \beta) \tri)$ is a log Calabi--Yau pair if the $\mathbb{R}$-Cartier divisor $K_X + (1 - \beta )\tri$ is $\mathbb{R}$-linearly equivalent to zero.
\end{defn}
In particular, $C_1 (X, \tri ) := C_1 (-K_X - (1 - \beta )\tri)$ makes sense and equals zero.

It was shown in \cite[Corollary 6.3]{MR3403730} that if $(X, (1 - \beta ) \tri)$ is a log Calabi--Yau pair, then $(X, (1 - \beta ) \tri)$ is log canonical if and only if $((X,L);  \tri)$ is log $K$-semistable with angle $2 \pi \beta$. 

Compared with \thmref{log Calabi Yau}, we actually further have
\begin{thm}\label{Log Calabi Yau}
Suppose that $(X, (1 - \beta ) \tri)$ is a log Calabi--Yau pair. 
If $(X, (1 - \beta ) \tri)$ is Kawamata log terminal, then $((X,L); \tri)$ is uniformly log $K$-stable with angle $2 \pi \beta$.
\end{thm}
\begin{proof}
Under the assumption that $K_X+(1-\beta) \tri$ is $\mathbb{R}$-linearly equivalent to zero, we have $\ul S_\beta=0$. Then \eqref{Algebraic log Donaldson-Futaki} becomes
\begin{align*}
DF (\mathcal B , \mathcal{B}_{\tri} ,\mathcal{L}' -E , \beta) =(\mathcal{L}'-E)^n\cdot K_e.
\end{align*}Therefore, $(2)$ and $(6)$ in \lemref{conditions} is applied to obtain that the log canonical (resp.~Kawamata log terminal) condition implies log $K$-semistability (resp.~uniform log $K$-stability).
\end{proof}

\subsubsection{Log canonical pair}
Actually, these statements could be extended to the case when $C_1(X,\tri )$ is not necessary vanishing, generalising \thmref{properness criteria thm}, \corref{negative c1} and \corref{positive c1} in Section \ref{Log stable manifolds} to singular varieties.

We recall the definition of a nef cohomology class.
\begin{defn}\label{algebraic nef}
An $\mathbb{R}$-Cartier divisor $F$ on a normal projective variety $X$ is said to be \textbf{nef} if for any irreducible curve $C$ in $X$ we have $C.F \ge 0$, or equivalently
\begin{equation*}
	\int_C C_1 (F) \ge 0.
\end{equation*}
\end{defn}

Observing that the nefness depends only on the numerical equivalence class of the divisor, we may abuse the terminology and say that the cohomology class $C_1 (F)$ is nef when $F$ is nef.

\begin{thm}\label{log canonical}
Suppose that $X$ is a $\mathbb{Q}$-Gorenstein normal projective variety, $\tri$ is an effective integral reduced Cartier divisor on $X$, and $(X , (1 - \beta ) \tri )$ is log canonical. 
We have the following two conclusions.
\begin{itemize}
\item Suppose $\ul S_\beta< 0$ and that the cohomology Condition \eqref{cohomology conditions Kahler} holds,  
\begin{align*}
\left\{
\begin{array}{lcl}
	&(i)&0\leq \eta<\frac{n+1}{n}\alpha_\beta,\\
	&(ii)& C_1(X,\tri)<\eta C_1(L) ,\\
	&(iii)& (\ul S_\beta-\eta) C_1(L) <(n-1)C_1(X,\tri).
\end{array}
\right.
\end{align*} 

Then $((X,L);\tri)$ is uniformly log $K$-stable with angle $2 \pi \beta$.
\item Suppose that the following conditions hold
	\begin{align}\label{cohomology conditions Kahler variety weaker}
	\ul S_\beta<(n+1)\alpha_\beta , \quad \text{and} \quad -\ul S_\beta C_1(L) + (n+1)C_1(X,\tri)\text{ is nef}.
	\end{align}
	
	Then $((X,L);\tri)$ is uniformly log $K$-stable with angle $2 \pi \beta$.
\end{itemize}
\end{thm}

Recall that $(X, (1 - \beta) \tri)$ being log canonical implies $\alpha_{\beta} \ge 0$.

\begin{rem} \label{remcpdervan}
Dervan \cite[Theorem 1.3]{MR3428958} proved that the condition \eqref{cohomology conditions Kahler variety weaker} implies that $((X,L); \tri)$ is log $K$-stable with angle $2 \pi \beta$. The second part of the theorem above strengthens this result to the uniform log $K$-stability; this may be well-known to the experts since the twisted version already appeared in \cite[Theorem 1.9]{MR3564626} and the proof that we present below is based on Dervan's \cite{MR3428958,MR3564626}, but the statement as above does not seem to have appeared in the literature to the best of the authors' knowledge (see \cite[Remark 3.26 and Theorem 3.27]{MR3564626} for the comparison between the log $K$-stability and the twisted $K$-stability).
\end{rem}

A direct corollary is given below.
\begin{cor}\label{log canonical corollary}
Suppose that
 \begin{align*}
C_1(X,\tri )< 0, \quad \text{and} \quad -\ul S_\beta C_1(L) + n C_1(X,\tri )\text{ is nef}.
\end{align*} If $(X, (1 - \beta) \tri)$ is log canonical, then $((X,L);\tri)$ is uniformly log $K$-stable with angle $2 \pi \beta$.
\end{cor}
\begin{proof}
When $C_1(X, \tri)<0$, the first two conditions in Condition \eqref{cohomology conditions Kahler}  hold for any sufficiently small $\eta>0$. The third condition is satisfied, similar to Remark \ref{conditions 13}. Precisely speaking, we have
\begin{align*}(\ul S_\beta -\eta) C_1(L) < n C_1(X,\tri)< (n-1) C_1(X,\tri), \quad \forall \eta>0
\end{align*} and $-\ul S_\beta C_1(L) +(n-1)C_1(X,D)$ is nef. 
\end{proof}
\subsubsection{Condition \eqref{cohomology conditions Kahler} and Condition \eqref{cohomology conditions Kahler variety weaker}}
Before we prove this theorem, we discuss these two conditions. When $X$ is smooth, it is clear that Condition \eqref{cohomology conditions Kahler variety weaker} is weaker than Condition \eqref{cohomology conditions Kahler}.
\begin{rem}\label{conditions 12}
The assumptions $(i,ii)$ in Condition \eqref{cohomology conditions Kahler}  together imply the upper bound of the average scalar curvature $$\ul S_\beta<n\eta<(n+1)\alpha_\beta.$$
\end{rem}
\begin{rem}\label{conditions 13}
We could reword $(iii)$ in Condition \eqref{cohomology conditions Kahler}  and it says 
\begin{align}\label{conditions 13 equation}
-\ul S_\beta C_1(L) +(n-1)C_1(X,D) \text{ is nef}.
\end{align}
\end{rem}

\begin{prop}\label{compare 2 conditions}
Suppose that $C_1(X,\tri)$ is nef and that $(X , (1 - \beta) \tri )$ is a log canonical pair. Then the cohomology Condition \eqref{cohomology conditions Kahler} implies Condition \eqref{cohomology conditions Kahler variety weaker}.
\end{prop}
\begin{proof}
The first condition is obtained in Remark \ref{conditions 12}. The second condition follows from Remark \ref{conditions 13} as well.
\end{proof}

\begin{rem}
Given a compact subgroup $P$ of the automorphism group $\mathrm{Aut}_0 ((X,L);D)$, we could restrict ourselves in the $P$-equivariant setting. It is natural to speculate that, by making use of $P$-equivariant log alpha invariant and $G$-equivariant log test configuration, we can obtain corresponding results for uniform log $K$-stability, similar to \thmref{log canonical}. We may then hope to compare (\ref{cohomology conditions Kahler variety weaker}) and Proposition \ref{ppcnaglfut} to find a manifold that does not have cscK cone metric. We decide not to pursue this point any further, however, in this paper.
\end{rem}



\subsection{Proof of \thmref{log canonical}}
We will need the following lemma.
\begin{lem}\label{conditions}\emph{(\cite[Lemma 4.2]{MR2889147}, \cite[Section 5]{MR3428958}, \cite[Lemma 3.10, Remark 3.11 and 3.17]{MR3564626})}

Suppose that $X$ is a $\mathbb{Q}$-Gorenstein normal projective variety, $\tri$ is an effective integral reduced Cartier divisor on $X$, and $(X , (1 - \beta ) \tri )$ is log canonical.

\begin{enumerate}
\item Let $R$ be a nef divisor on $X$ and $\mathcal R= p^\ast R$, where $p$ is the composition of the blowdown map $\pi : \mathcal{B} \to X \times \mathbb{P}^1$ and the natural projection $X \times \mathbb{P}^1 \to X$. Then $(\mathcal{L}'-E)^n\cdot \mathcal R\leq 0. $
\item $(\mathcal{L}'-E)^n\cdot K_{e}\geq 0$.

\item The exceptional divisor $K_e- \alpha_\beta E$ is effective and $(\mathcal{L}'-E)^n(K_e- \alpha_\beta E)\geq 0$.
\item $(\mathcal{L}'-E)^n E>0$.
\item $(\mathcal{L}'-E)^n(\mathcal{L}'+nE)=(n+1)\|\mathcal B\|_m$ and $(\mathcal{L}'-E)^n E\geq\frac{n+1}{n} \Vert \mathcal B \Vert_m$.
\item $(\mathcal{L}'-E)^n\cdot K_e\geq \alpha_\beta\frac{n+1}{n} \Vert \mathcal B \Vert_m $.

\end{enumerate}
Recall that in the above $K_e$ stands for the exceptional part of the divisor $K_{\mathcal B/(X,(1-\beta)\tri)\times\mathbb P^1}$ defined in (\ref{dfrellgcb}), and that $\Vert \cdot \Vert_m$ stands for the minimum norm in (\ref{dfminnorm}).
\end{lem}

\begin{proof}
	We simply indicate where the proof can be found in the literature and omit the details.
	
	The first and the fourth item is exactly as written in \cite[Lemma 3.10]{MR3564626}; see also \cite[Lemma 3.6]{MR3428958}. The second and the third item follows form \cite[equations (91)--(94)]{MR3428958} and \cite[Lemma 3.6]{MR3428958}. The fifth and the sixth item is exactly as written in \cite[page 4770]{MR3564626} (where we also use the third item above).
\end{proof}

Now we start to prove \thmref{log canonical}.

\begin{proof}[Proof of \thmref{log canonical}]
We rewrite the log Donaldson--Futaki invariant \eqref{Algebraic log Donaldson-Futaki} as
\begin{align*}
&DF (\mathcal B , \mathcal{B}_{\tri},\mathcal{L}' -E , \beta )\\
 & =(\mathcal{L}'-E)^n\cdot \{(\ul S_\beta-\eta)\mathcal{L}'+(n-1)[\pi^* ((K_X+(1-\beta)\tri) \times \mathbb{P}^1)]\\
 &-\frac{\ul S_\beta }{n+1}(n\mathcal{L}'+E)+(n-2)\eta \mathcal{L}'+K_e\\
 &+(2-n)[\eta \mathcal{L}'+\pi^* ((K_X+(1-\beta)\tri) \times \mathbb{P}^1)]\}.
\end{align*}
According to $(1)$, $(2)$ and $(3)$ in \lemref{conditions}, we have
\begin{align*}
&DF (\mathcal B , \mathcal{B}_{\tri} ,\mathcal{L}' -E , \beta  )\\
&\geq (\mathcal{L}'-E)^n\cdot \{ -\frac{\ul S_\beta }{n+1}(n\mathcal{L}'+E)+(n-2)\eta \mathcal{L}'+K_e\},\end{align*}
by noting $\eta \mathcal{L}'+\pi^* ((K_X+(1-\beta)\tri) \times \mathbb{P}^1) = p^* (\eta L + K_X+(1-\beta)\tri)$. We set $e=-\frac{n}{n+1}\ul S_\beta+(n-2)\eta$ and $f=-\frac{\ul S_\beta }{n+1}$. Then
\begin{align*}
DF (\mathcal B , \mathcal{B}_{\tri} ,\mathcal{L}' -E , \beta  )&
\geq (\mathcal{L}'-E)^n[e\mathcal{L}'+fE+K_\eps]\\
& =  (\mathcal{L}'-E)^n[e(\mathcal{L}'+nE)+(f-en)E+K_e].
\end{align*}
At last, $(5)$ in \lemref{conditions} implies that
\begin{align*}
&DF (\mathcal B , \mathcal{B}_{\tri} ,\mathcal{L}' -E , \beta )\\&
\geq e(n+1)\|\mathcal B\|_m+(f-en)\frac{n+1}{n} \Vert \mathcal B \Vert_m+(\mathcal{L}'-E)^nK_\eps\\
&
= -\frac{\ul S_\beta }{n} \Vert \mathcal B \Vert_m+(\mathcal{L}'-E)^nK_\eps.
\end{align*}
The first conclusion in the theorem is proved directly, since $\ul S_\beta<0$ and $(\mathcal{L}'-E)^nK_e\geq 0$ from $(2)$ in \lemref{conditions}.

The proof of the second conclusion is an analogue of \cite[Theorem 1.9]{MR3564626} for the twisted case and the result regarding log $K$-stability was proven in \cite[Theorem 1.3]{MR3428958}.

Applying $(1)$ and $(3)$ in \lemref{conditions}, with the hypothesis that $ -\ul S_\beta \Om+ (n+1)C_1(X,\tri)$ is nef, and the exceptional divisor $K_e- \alpha_\beta E\geq0$, we see that
\begin{align*}
&DF (\mathcal B , \mathcal{B}_{\tri} ,\mathcal{L}' -E , \beta )\\
 & =(\mathcal{L}'-E)^n\cdot \left[ \frac{\ul S_\beta}{n+1} \mathcal{L}'+ \pi^* ((K_X+(1-\beta)\tri) \times \mathbb{P}^1) -\frac{\ul S_\beta}{n+1} E+K_e \right]\\
 & =(\mathcal{L}'-E)^n\cdot p^* \left( \frac{\ul S_\beta}{n+1} L+ K_X+(1-\beta)\tri \right) +(\mathcal{L}'-E)^n\cdot  \left( \frac{-\ul S_\beta}{n+1} E+K_e \right) \\
 &\geq (\mathcal{L}'-E)^n\cdot \left( \alpha_\beta-\frac{\ul S_\beta}{n+1} \right) E 
 \geq \frac{n+1}{n} \left( \alpha_\beta-\frac{\ul S_\beta}{n+1}  \right) \|\mathcal B\|_m,
\end{align*}which completes the proof, by $(5)$ in \lemref{conditions}.
\end{proof}

\begin{rem}
It is shown in \cite[Theorem 5.7]{MR3428958} that a log canonical $(X, (1 - \beta) \tri)$ satisfying \eqref{cohomology conditions Kahler variety weaker} is log $K$-stable with cone angle $2 \pi \beta$. The proof is identical to the one given above, where the last inequality is strict, since $(4)$ is used instead of $(5)$ in \lemref{conditions}, that is
\begin{align*}
DF (\mathcal B , \mathcal{B}_{\tri} ,\mathcal{L}' -E ,\beta )\geq (\mathcal{L}'-E)^n\cdot (\alpha_\beta-\frac{\ul S_\beta}{n+1} ) E
>0.
\end{align*}
\end{rem}


\subsection{From log semi-stability to singularity types of varieties}

Let $X$ be a $\mathbb{Q}$-Gorenstein normal projective variety and $\tri$ be an effective integral $\mathbb{Q}$-Cartier divisor, as before. The hypotheses on $X$ can be much weakened (see \cite[Definition 1.1]{MR3010808} and \cite[Section 6]{MR3403730}), but we content ourselves with the ones that we have been working with so far.

\begin{thm}\label{Kawamata log terminal inverse}
Suppose that $((X,L);\tri)$ is log $K$-semistable with angle $2 \pi \beta$ and that we have
\begin{align*}
C_1(X,\tri )>0, \quad \ul S_\beta C_1(L) - nC_1(X, \tri)\text{ is nef}.
\end{align*}
Then $(X, (1 - \beta) \tri)$ is Kawamata log terminal.
\end{thm}

Note also that the log $K$-semistability of $((X,L);\tri)$ implies that $(X, (1 - \beta) \tri)$ is log canonical, by a result of Odaka--Sun who proved a more general version \cite[Theorem 6.1]{MR3403730}. They also proved that $(X, (1 - \beta) \tri)$ is Kawamata log terminal if $L$ is a positive multiple of $-K_X - (1 - \beta )\tri $; the result above relaxes this hypothesis on $L$.

\begin{proof}
We essentially repeat the argument in \cite[Proof of Theorem 6.1]{MR3403730} and \cite[Proof of Theorems 3.28 and 3.30]{MR3564626}, which the reader is referred to for the details, and only provide a brief summary here for the reader's convenience.

The formula \eqref{Algebraic log Donaldson-Futaki} gives
\begin{align*}
DF (\mathcal B , \mathcal{B}_{\tri} ,\mathcal{L}' -E , \beta )
 &=(\mathcal{L}'-E)^n\cdot [-\frac{\ul S_\beta}{n(n+1)} (\mathcal{L}'+nE)\\
 &+\frac{\ul S_\beta}{n}\mathcal{L}'+\pi^* ((K_X+(1-\beta)\tri) \times \mathbb{P}^1) +K_e].
\end{align*}
Then, applying $(1)$ and $(5)$ in \lemref{conditions} with the stated assumptions and noting $\frac{\ul S_\beta}{n}\mathcal{L}'+\pi^* ((K_X+(1-\beta)\tri) \times \mathbb{P}^1) = p^* (\frac{\ul S_\beta}{n} \Omega + K_X+(1-\beta)\tri)$, we get
\begin{align*}
DF (\mathcal B , \mathcal{B}_{\tri} ,\mathcal{L}' -E , \beta )
<(\mathcal{L}'-E)^n\cdot K_e.
\end{align*}
So, the proof is completed once there exists a flag ideal $\mathcal I$ with $K_e=0$ when $(X , (1- \beta) \tri )$ is not Kawamata log terminal, which can be done exactly as in \cite[Proof of Theorem 6.1]{MR3403730}.
\end{proof}


%

\subsection{Uniform log $K$-stability for standard polarised pair}
We will apply \eqref{Algebraic log Donaldson-Futaki} to check uniform log $K$-stability for standard polarised pair $((X,L);D,m)$, which was defined in Definition \ref{canonical polarised pair}. In particular, we assume that $X$ and $D$ are both smooth in the rest of this section, and moreover $D$ is chosen to be sitting in the linear system $|mL|$.

For uniform log $K$-stability, the following \thmref{topological conditions to beta conditions} extends \thmref{small angle existence} for existence of cscK cone metrics.
We will also present results related to Question \ref{cone angle comparison}. We have a lower bound of $\beta_{ulKs}$, which was defined in Definition \ref{maximal uniformly log K-stable}, i.e.~the maximal angle $\beta_{ulKs}$ such that $((X,L);D)$ is uniformly log $K$-stable with angle $2\pi\beta$.

\begin{thm}\label{topological conditions to beta conditions}
Let $((X,L);D,m)$ be a standard polarised pair (Definition \ref{canonical polarised pair}). 
Suppose that 
\begin{align}\label{s1 m n}
\ul S_1\leq mn\text{ and }(n+1)\lambda\leq \ul S_1+m.
\end{align} 

Then $((X,L);D,m)$ is uniformly log $K$-stable with cone angle $2\pi\beta$ satisfying the constraint 
\begin{align}\label{angle conditions 2 equation}
1-\frac{(n+1)\lambda-\ul S_1}{m}\leq \beta<\beta_u.
\end{align}
	Here, $\beta_u$ was defined in Definition \ref{beta u} and we fixed the constants $\lambda$ and $\Lambda$ as in Definition \ref{Lambda} so that it satisfies $ \lambda \Omega\leq C_1(X) \le \Lambda \Omega$.

Consequently, we also have
\begin{align*}
\beta_{ulKs}\geq \beta_u.
\end{align*}

\end{thm}
\begin{proof}
$(X , (1 - \beta ) D )$ being log canonical necessarily implies $1 - \beta \le 1$, and hence $\beta \ge 0$. We also need $1 - \beta \ge 0$ since $D$ is effective. From \eqref{angle conditions 2 equation}, we have $\beta\geq 0$ by using $(n+1) \lambda\leq \ul S_1+m$, and $\beta\leq 1$ by the definition of $\beta_u$.

Then it is sufficient to check the condition \eqref{cohomology conditions Kahler variety weaker} in \thmref{log canonical}.
From the following two lemmas, \lemref{conditions for mL} and \lemref{angle conditions 2}, we have
$$\text{\eqref{angle conditions 2 equation}} \Longrightarrow \text{\eqref{angle conditions}} \Longrightarrow \text{\eqref{cohomology conditions Kahler variety weaker}}.$$ Therefore, \thmref{log canonical} implies the required result.
\end{proof}
\begin{rem}
It is interesting to compare conditions \eqref{s1 m n}, \eqref{angle conditions 2 equation} in \thmref{topological conditions to beta conditions} with the conditions \eqref{given m} and \eqref{case 2 angle} in \thmref{small angle existence}. It may suggest certain chance to find some parameters such that some standard polarised pair $((X,L);D,m)$ does not admit s cscK cone metric, but is uniformly log $K$-stable with angle $2\pi\beta$.
\end{rem}

We need to interpret the topological condition \eqref{cohomology conditions Kahler variety weaker} in \thmref{log canonical} into a cone angle constraint.
\begin{lem}\label{conditions for mL}
The condition \eqref{cohomology conditions Kahler variety weaker} is deduced from the following bound of the cone angle $\beta$,
\begin{align}\label{angle conditions}
1-\frac{(n+1)\lambda-\ul S_1}{m}\leq \beta<1-\frac{\ul S_1}{mn}+\frac{n+1}{mn}\alpha_\beta.
\end{align}
\end{lem}
\begin{proof}
We recall by \lemref{lmlgmprp preliminary} that $C_1(X,D)=C_1(X)- m (1-\beta) C_1(L)$ and 
\begin{align*}
\ul{S}_\beta=\ul S_1-mn(1-\beta), \quad \ul S_1=n \frac{\int_X C_1 (X) C_1(L)^{n-1}}{\int_X C_1(L)^n}.
\end{align*}
We then check the condition \eqref{cohomology conditions Kahler variety weaker}.
The first one becomes,
\begin{align*}
\ul{S}_\beta=\ul S_1-mn(1-\beta)<(n+1)\alpha_\beta,
\end{align*}
which means 
\begin{align*}
\beta<1-\frac{\ul S_1}{mn}+ \frac{n+1}{mn}\alpha_\beta.
\end{align*}
While the second condition says for all $\eta>0$,
\begin{align*}
(n+1)[C_1(X)- m (1-\beta) C_1(L) ]+\eta C_1(L) > [\ul S_1-mn(1-\beta)] C_1(L),
\end{align*} that is
\begin{align}\label{beta form}
 [m (1-\beta) +\ul S_1-\eta] C_1(L) < (n+1)C_1(X).
\end{align} 
With the lower bound of $C_1(X)$,
$
C_1(X)\geq \lambda\cdot C_1(L),
$ 
the inequality above is strengthened to 
\begin{align*}
\beta > 1-\frac{(n+1)\lambda-\ul S_1+\eta}{m}.
\end{align*} Thus the lower bound of $\beta$ is obtained, as $\eta\rightarrow 0$.
\end{proof}
\begin{rem}
According to the definition \eqref{Lambda}, we have $n\Lambda\geq \ul S_1$, so the left-hand side of \eqref{angle conditions}
\begin{align*}
1-\frac{(n+1)\lambda-\ul S_1}{m}\leq 1-\frac{(n+1)\lambda-n\Lambda}{m}.
\end{align*}
\end{rem}
\begin{rem}
If $C_1(X)\leq 0$, the constant $\lambda$ is non-positive and we have
\begin{align*}
\beta\geq 1+\frac{\ul S_1}{m}.
\end{align*}
In the case when $C_1(X)>0$, we multiply \eqref{beta form} with $C_1(L)^{n-1}$ and make use of the definition of $\ul S_1$ to see
\begin{align*}
\beta\geq1-\frac{\ul S_1}{mn}.
\end{align*} 
\end{rem}

\begin{lem}\label{angle conditions 2}
Suppose $\ul S_1\leq mn$. Then the condition \eqref{angle conditions 2 equation} implies the condition \eqref{angle conditions}.
\end{lem}
\begin{proof}
We only need to check the right-hand side of \eqref{angle conditions} .
Recall that the slope $\mu$ defined in \eqref{dfslopemu} satisfies $\mu=\frac{\ul S_1}{n}-m(1-\beta)$. Then \eqref{angle conditions} becomes
\begin{align}\label{mu bound}
-n\lambda+\ul S_1\leq \frac{n}{n+1}\mu<\alpha_\beta.
\end{align}
The lower bound of the log alpha invariant \eqref{log alpha invariant lower bound} gives 
\begin{align*}
		\alpha_\beta \geq  \min\{ m\beta, \alpha(L), m\alpha(L_D\vert_D) \}.
\end{align*}
Under the assumption of $m$, i.e. $\ul S_1\leq mn$, we have $m\beta>\frac{n}{n+1}\mu$. Therefore, we obtain the conditions in \thmref{topological conditions to beta conditions}, as in the argument in \thmref{small angle existence}.
\end{proof}
At last, we link \thmref{small angle existence} and \thmref{topological conditions to beta conditions} to known results on K\"ahler--Einstein cone metrics on Fano manifolds.
\begin{cor}\label{KE beta uniform}Let $((X,L);D,m)$ be a standard polarised pair with $m\geq 1$. 
Suppose $C_1(X)>0$ and $L=-K_X$, i.e.~$X$ is Fano. If $0<\beta<\beta_u$ with 
\begin{align*}
\beta_{u}(-K_X)=\min\{1, 1-\frac{1}{m}+\frac{n+1}{mn}\min\{ \alpha(-K_X), \alpha(L_D\vert_D) \}\},
\end{align*} then 
\begin{enumerate}
\item $((X,-K_X);D,m)$ is uniformly log $K$-stable with angle $2\pi\beta$. Moreover, the maximal existence angle $\beta_{cscKc}$ satisfies that
\begin{align*}
\beta_{cscKc}\geq \beta_{u}.
\end{align*}
\item The log $K$-energy is proper.
$((X,-K_X);D,m)$ admits a cscK cone metric, which is a K\"ahler--Einstein cone metric satisfying
\begin{align}\label{KE cone metric}
Ric(\om)=\mu\cdot\om+2\pi(1-\beta)[D], \quad \mu=1-m(1-\beta).
\end{align}

\end{enumerate}
\end{cor}
\begin{proof}Under the assumption, we have $\Lambda=\lambda=1$ and $\ul S_1=n$. Then $\ul S_\beta=n-mn(1-\beta)=n\mu$ and the cone angle constrain \eqref{angle conditions} becomes 
\begin{align*}
1-\frac{1}{m}\leq \beta<1-\frac{1}{m}+\frac{n+1}{mn}\alpha_\beta.
\end{align*} 
We get an equivalent inequality 
\begin{align*}
0\leq \frac{n}{n+1}\mu<\alpha_\beta.
\end{align*} 
By \eqref{log alpha invariant lower bound}, we know 
\begin{align*}
		\alpha_\beta \geq  \min\{ m\beta, \alpha(-K_X), m\alpha(L_D\vert_D) \}.
\end{align*}
Clearly, $m\beta\geq \frac{n}{n+1}\mu$. So \thmref{topological conditions to beta conditions} implies that
if 
\begin{align*}
\min\{ \alpha(-K_X), m\alpha(L_D\vert_D) \}> \frac{n}{n+1}\mu,
\end{align*} then $((X,-K_X);D,m)$ is uniformly log $K$-stable. The first conclusion follows from rewriting this condition in $\beta$.

The properness in second conclusion follows from Proposition \ref{properness criteria}, if we show that the condition \ref{cohomology conditions Kahler variety weaker} implies the Condition \eqref{cohomology conditions Kahler}. Actually, under the assumption, we have $C_1(X,D)=\mu C_1(X)$. Then \eqref{cohomology conditions Kahler variety weaker} becomes
	\begin{align*}
	n\mu<(n+1)\alpha_\beta,\quad -n\mu + (n+1)\mu \geq 0.
	\end{align*}
 While, the topological Condition \eqref{cohomology conditions Kahler} reads
$$
	0\leq \eta<\frac{n+1}{n}\alpha_\beta,\quad
	 \mu <\eta,\quad
	 (n\mu-\eta)<(n-1)\mu .
$$
We could choose $\eta=\mu+\eps$ with some sufficiently small $\eps>0$ and versify that all these three inequalities are satisfied. Thus, \eqref{cohomology conditions Kahler variety weaker} gives Condition \eqref{cohomology conditions Kahler}.

Furthermore, the existence \thmref{properness criteria thm} implies there exists cscK cone metric in $\Om$, see also \thmref{small angle existence}. Actually, it is K\"ahler--Einstein cone metric \eqref{KE cone metric} by uniqueness, \thmref{uniqueness of cscK cone metric} (see \cite{MR3761174} and references therein for more results on K\"ahler--Einstein cone metric).
\end{proof}

\begin{cor}\label{KE cone uniform stability}
Suppose that there exists $D \in |-K_X|$ which is smooth, and take $m=1$ in \corref{KE beta uniform}. Then
we have
$$0\leq \beta<\min\{1, \frac{n+1}{n}\min\{ \alpha(-K_X), \alpha(L_D\vert_D) \}\}.$$
Consequently, the maximal angle for uniform log $K$-stability (Definition \ref{maximal uniformly log K-stable}) has lower bound
\begin{align*}
\beta_{ulKs}\geq \min\{1, \frac{n+1}{n}\min\{ \alpha(-K_X), \alpha(L_D\vert_D) \}\}.
\end{align*} 
\end{cor}
\begin{rem}
For K\"ahler--Einstein cone metric, the uniform log $K$-stability is also equivalent to the log $K$-stability, as a result of the resolution of the YTD conjecture for the K\"ahler--Einstein metric. In general, these two stabilities are not expected to be equivalent.
\end{rem}
\begin{rem}
Twisted K\"ahler--Einstein path $\om\in C_1(X)$ is defined to be
\begin{align*}
Ric(\om)=\beta\cdot\om+(1-\beta)\om_0,
\end{align*}with a smooth K\"ahler metric $\om_0\in C_1(X)$ and $0\leq \beta\leq 1$, 
the lower bound of the maximal existence time $\beta$ of the twisted K\"ahler--Einstein path was introduced in \cite{MR2771134} and its lower bound was given in \cite[Corollary 3.2]{MR3564626}.
\end{rem}









\end{document}